\documentclass[11pt]{amsart}

\AtBeginDocument{%
   \def\MR#1{}
}
\usepackage{tasks}
\usepackage{hyperref}
\usepackage{blindtext}

\usepackage{graphicx}
\usepackage{subcaption}
\usepackage[utf8]{inputenc}
\usepackage[english]{babel}
\usepackage[export]{adjustbox}
\usepackage{wrapfig}

\usepackage{marginnote}
\usepackage{datetime}
\usepackage{color}
\usepackage{amsrefs,amsmath,amsthm}
\usepackage{mathtools}

\usepackage{setspace}
\usepackage[letterpaper,left=1in,right=1in,top=1in,bottom=1in]{geometry}
\usepackage{hyperref}
\hypersetup{
	colorlinks = true,
	linkcolor = {magenta},
	citecolor = {blue}
}

\newtheorem{thm}{Theorem}[section]
\newtheorem{lem}[thm]{Lemma}
\newtheorem{prop}[thm]{Proposition}

\theoremstyle{definition}

\theoremstyle{remark}

\numberwithin{equation}{section}

\usepackage{paralist}
 
\theoremstyle{definition}

\theoremstyle{remark}
\newtheorem*{remark}{Remark}

 \usepackage{bm,amsmath}

 \DeclareMathOperator*{\esssup}{ess\,sup}
\newcommand{\RN}[1]{%
  \textup{\uppercase\expandafter{\romannumeral#1}}%
}

\newcommand{\bff}{\mathbf{f}}
\newcommand{\bfx}{\mathbf{x}}
\newcommand{\bfu}{\mathbf{u}}
\newcommand{\bfe}{\mathbf{e}}
\newcommand{\D}{\mathrm{D}}   
\newcommand{\ddt}{\frac{\rm d}{\rm dt}}
   
\newcommand{\bfv}{\mathbf{v}}
\newcommand{\bfw}{\mathbf{w}}

\newcommand{\bfV}{\mathbf{V}}

\def \lb {\langle}
\def \rb {\rangle}
\newcommand{\bfn}{\mathbf{n}}

\newcommand{\vbf}{\mathbf{v}}
\newcommand{\ubf}{\mathbf{u}}

\usepackage{subcaption}
\newenvironment{Andrearev}{\color{blue}}{\color{black}}
\newcommand{\BBB}{\begin{Andrearev}}
\newcommand{\FFF}{\end{Andrearev}}

\allowdisplaybreaks[4]

\begin{document}

\date{\today}

\title[Data Assimilation for the 3D Ladyzhenskaya model]{Continuous data assimilation for the 3D Ladyzhenskaya model: analysis and computations}


\author{Yu Cao}
\address{Department of Mathematics, Florida State University, FL 32304, USA}
\email{ycao2@fsu.edu}
\author{Andrea Giorgini}
\address{Department of Mathematics, Indiana University Bloomington, IN 47405, USA}
\email{agiorgin@iu.edu}
\author{Michael Jolly}
\address{Department of Mathematics, Indiana University Bloomington, IN 47405, USA}
\email{msjolly@indiana.edu}
\author{Ali Pakzad}
\address{Department of Mathematics, Indiana University Bloomington, IN 47405, USA}
\email{apakzad@iu.edu}



\maketitle
\setcounter{tocdepth}{1}

\begin{abstract}

We analyze continuous data assimilation by nudging for the 3D Ladyzhenskaya equations. The analysis provides conditions on the spatial resolution of the observed data that guarantee synchronization to the reference solution associated with the observed, spatially coarse data. This synchronization holds even though it is not known whether the reference solution, with initial data in $L^2$, is unique; a particular reference solution is determined by the observed, coarse data.  The efficacy of the algorithm in both 2D and 3D is demonstrated by numerical computations. 

\end{abstract}

\section{Introduction}

\subsection{Data Assimilation} 
The insertion of coarse grain observational measurements into a mathematical model is called continuous data assimilation.  This can provide a more accurate forecast in  applications ranging from the medical, environmental and biological sciences,  \cites{Kost2, Kost-1}, to imaging, traffic control, finance and oil exploration \cite{Asch-Data2016}. 
Bayesian and variational approaches (Kalman filters, 3DVar and 4DVar) are based on discrete observations in time and often used to treat errors in both observed data and model itself \cites{Blomker-Accuracy2013, Branicki-AnInformation2015, 
Branicki-Accuracy2018, Majda-Filtering2012,Kalnay-Atmospheric2003,Kelly-Well2014,Law-Data2015,Reich-Probabilistic2015}.   They are widely used in practice, but difficult to analyze mathematically, especially for physical models governed by nonlinear differential equations \cites{Harlim-Catastrophic2010,Tong-1Nonlinear2016, Tong-2Nonlinear2016}.

Nudging is a straightforward, deterministic  approach to data assimilation.  While its origin can be traced back to \cite{Hoke-TheInitialization1976}, it has been more recently applied in the context of synchronizing chaotic dynamical systems.  See \cites{Auroux-ANudging2008, Pecora-Synchronization2015} for a more complete history, and \cite{Vidard-Determination2003} for a comparison with Kalman filtering.  In essence, this method assumes that an accurate initial condition $\bfu(0)$ is not known for a particular model
$$ 
\frac{\mathrm{d} \bfu}{\mathrm{d}t} =\mathbf{F}(\bfu),
$$
but data from a reference solution, interpolated at spatial resolution $h$ is available, denoted as $I_h \mathbf{u}(t)$. Those observations are used in an auxiliary system
\begin{align}\label{nudgedODE} 
\frac{\mathrm{d} \bfv}{\mathrm{d}t} =\mathbf{F}(\bfv) -\mu I_h(\bfv-\bfu)\;, \quad \bfv(0)=0
\end{align}
to drive $\|\bfv-\bfu\| \to 0$ at an exponential rate, provided $\mu$ is sufficiently large and $h$ is sufficiently small.  Since derivatives are not required of the data in the approach, it can be used with more common types of observations, such as nodal values.  

Rigorous analysis of the nudging algorithm for partial differential equations in fluid mechanics began with the work of Azouani, Olson and Titi.  They estimated threshold values for the relaxation parameter $\mu$ and data resolution $h$ for the  2D NSE \cite{Azouani-Continuous2014}.
 Since then, nudging has been rigorously shown to synchronize with reference solutions in a variety of applications, including the 2D Rayleigh-B\'enard problem \cites{Farhat-Continuous2015, Farhat-Continuous2017, Farhat-Data2020}, surface quasigeostrophic equation \cites{Jolly-Adata2017,Jolly-Continuous2019},  Korteweg--de Vries equation \cite{Jolly-Determining2017}, 2D  magnetohydrodynamic system \cite{Biswas-Continuous2018}, 3D Brinkman-Forchheimer-extended Darcy model \cite{Markowich-Continuous2019}, and 3D primitive equations \cite{Pei-Continuous2019}.  In each case, threshold values for $\mu$ and $h$ had to be established for both the well-posedness of the corresponding system \eqref{nudgedODE} as well as for synchroniziation. In some works it has been shown that it is sufficient to nudge with data in only a subset of the system variables
\cites{Farhat-Continuous2015, Farhat-Continuous2017, Farhat-Data2020, Farhat-Abridged2016, Farhat-Data2016}.    While the nudging algorithm does not lend itself to directly treat error in the model, the effect of error in the observed data has been studied in \cites{Bessaih-Continuous2015, Jolly-Continuous2019}.

\subsection{The Ladyzhenskaya model} 
The motion of an homogeneous, incompressible, viscous fluid in a domain $\Omega \subset \mathbb{R}^3$ is classically described by the momentum equation and the incompressibility constraint, that read as
\begin{equation}
 \label{motion}
 \begin{split}
\partial_t \bfu+   (\bfu \cdot \nabla) \bfu   - \nabla \cdot  \mathbf{T}(\D \bfu) \ +  \nabla P  & = \bff,\\ 
\nabla \cdot \bfu  &= 0,
\end{split}
\qquad \qquad \qquad \text{in } \, \Omega \times (0,\infty),
\end{equation}
where $\bfu$ is the fluid velocity, $P$ is the fluid pressure. Here, $\D \bfu $ denotes the symmetric part of the gradient of $\bfu$. In particular, the Navier-Stokes model corresponds to the case of Newtonian fluids characterized by the (linear) Stokes' law $\mathbf{T}(\D \bfu)= \nu_0 \D \bfu$. The lack of a global regularity result makes the analysis of the nudging algorithm problematic for the 3D NSE, though a recent work provides a condition on observed data which deals with this issue  \cite{biswas2020continuous}. In this work we consider a family of 3D globally well-posed modified Navier–Stokes equations, namely the Ladyzhenskaya  model. In the mid-1960s,  a number of modifications to the Navier-Stokes equations were suggested by  Ladyzhenskaya  for the description of the dynamics of viscous fluids when velocity gradients are large \cites{L67,L69,L98}. These equations form an important mathematical model describing the flow behavior of a wide class of non-Newtonian fluids \cites{MNRR, MR2005,MR1348587}.  In this work we consider one particular model (see equations  \eqref{Lady}), where the Cauchy tensor in \eqref{motion} takes the following nonlinear form
\begin{equation}
\label{T-nonlinear}
\mathbf{T} (\D \bfu)=   - \nabla \cdot  \left( 2\nu_0  + 2\nu_1 |\D \bfu|_F^{p-2} \right)\, \D \bfu.
\end{equation}
The above  relation is commonly used for non-Newtonian fluids with shear dependent viscosity, i.e. the dynamic viscosity depends on $|\D \bfu|_F^2$.  The model corresponding to  $p=2$  reduces to the Navier-Stokes equations (NSE) with kinematic viscosity equal to $\left( \nu_0+\nu_1\right)$.   For $p=3$,  it is mathematically equivalent to  the Smagorinsky model \cite{S63} and the NSE with the von Neumann Richtmyer artificial viscosity for shocks \cite{VR50}.

There are various reasons to consider  the Ladyzhenskaya models instead of the Navier-Stokes equations.  In the first place,  the laws of conservation of mass and momentum provide an undetermined system of partial differential equations for the velocity, pressure, and stress tensor. In general, this system of equations is not closed until the stress tensor, which represents all the internal forces, is related to the fluid velocity. Internal forces, and therefore also the stress tensor, must depend on local velocity differences  and some combination of derivatives of velocity, i.e., the deformation tensor. The simplest relation is a linear law between stress and deformation, which leads to the Navier-Stokes equations, see \cite{S59} for details.  This linear relation is only an approximation for a real fluid and schematically the  stress and deformation are related nonlinearly,  especially for large deformations.  A specific  nonlinear mathematical relationship between the stress and deformation can be derived from Stokes' hypotheses\footnote{Stokes introduced a series of requirements which together serves to define a ordinary fluid such as water and air \cite{S59}.}. If one retains some of  these nonlinear terms, we arrive at the Ladyzhenskaya model considered here. We refer the interested reader to \cites{MNRR, MR2005,MR1348587} for more details.  

Secondly and more  from a practical engineering point of view, the study of the Ladyzhenskaya equations are related to the field of turbulence modeling.  For some values of $p$, such as $p=3, 4$, the Ladyzhenskaya model considered here is equivalent to those of  popular turbulence models,  such as large eddy simulation (LES) and zero-equation models. In both applications and turbulence modeling, the behavior of averaged quantities are most important and often simulated. 
To do so,  the quantities describing the flow are decomposed into its averaged and fluctuating quantities. However,  averaging the  NSE yields a non-closed system; to close the system, one must provide the relationship between the fluctuating and the averaged quantities. There are a  wide range of closure assumptions which are known collectively as turbulence closure models. Two examples that are  widely used  are the  zero equation model (or algebraic model) and large eddy simulation, for more details see e.g., \cites{P00, L92}. The main feature of these models is that the non-closed part  (known as the Reynolds stresses), which represents the contribution of small scales in the system,  is related to the derivatives of the averaged quantities.  See \cites{BIL06, J04} for more details on the mathematics of large eddy simulation. 

Finally,  from the theoretical point of view, while the well-posedness has not been proven for the Navier-Stokes equations in three space dimensions, several results of existence, uniqueness and regularity of global-in-time solutions of the Ladyzhenskaya model have been proved in the last decades \cites{L67,L68,L69,L98}. This  provides a firm mathematical foundation for the study of (\ref{Lady}). 

\subsection{Results in this paper}
In this paper we develop a comprehensive study based on the theoretical analysis and large eddy simulation of the nudged system \eqref{nudgedODE} corresponding to the Ladyzhenskaya model \eqref{motion}-\eqref{T-nonlinear} (see \eqref{Lady-DA}) with both no-slip and periodic boundary conditions. 

In the no-slip case, we first use the Schauder fixed point theorem to prove that the nudged system has a unique global weak solution with $\bfu_0 \in L^2_\sigma(\Omega)$ provided $p\ge 5/2$.  Unlike some treatments of the nudged system for other models, this approach does not require $\mu$ to be large, nor $h$ to be small.  We then find a threshold value $\mu^*$ in terms of  $p$, $\nu_0$, $\nu_1$,
domain size and the Grashof number (see \eqref{Grashof}), such that for $\mu \ge \mu^*$ and $h$ correspondingly small, synchronization is guaranteed.  In the case of periodic boundary conditions, the existence of global weak solutions to the Ladyzhenskaya model over the wider range $p\ge11/5$, originating from $\bfu_0 \in L^2_\sigma(\Omega)$, has been established in \cite{MNRR}.  These weak solutions are not known to be unique, unless the initial condition is more regular (see \cite{MNRR}*{Theorem 4.37} for $\bfu_0 \in H^1_\sigma(\Omega)$, and also \cite{bulivcek2019uniqueness} for $\bfu_0 \in W^{1,p}_\sigma(\Omega)$ in the no-slip case).
Nonetheless, any one such weak solution becomes more regular after some time.  We prove, for the endpoint case $p=11/5$, a time averaged bound in  $L^{\frac{11}{5}}(\bar{t},\infty;W^{1,\frac{33}{5}}(\Omega))$ for the solution to the Ladyzhenskaya model, where  $\bar t$ is suitably chosen. That bound is then used to prove the synchronization of the nudged solution to the solutions of the Ladyzhenskaya model. In contrast to the previous cases studied in literature when uniqueness of the reference solution holds, the novelty of our result is that
synchronization takes place even without uniqueness of the reference solution for $11/5 \le p \le 5/2$.   
More precisely, for each reference solution $\bfu$ corresponding to an initial datum $\bfu_0 \in L^2_\sigma(\Omega)$, the nudged solution $\bfv_{\bfu}$ converges to $\bfu$ at an  exponential rate for large time. As a consequence of our analysis, it is worth concluding that if two reference solutions $\bfu$ and $\widetilde{\bfu}$ are such that $I_h(\bfu)(t)=I_h(\widetilde{\bfu})(t)$ 
as $t\to \infty$, then $\| \bfu-\widetilde{\bfu}\|\rightarrow 0$ as $t\rightarrow \infty$, i.e., the model has a finite number of {\it determining modes} for $11/5 \le p$.

We demonstrate the efficacy of the algorithm by
extensive computational studies. Numerical work with other fluid systems has shown that the nudging algorithm achieves synchronization with data that is much more coarse than required by the rigorous estimates \cites{Farhat-Assimilation2018,Altaf-Downscaling2017,Hudson-Numerical2019,Gesho-Acomputational2016}. We find this is also the case for the Ladyzhenskaya model with periodic boundary conditions, for which we achieve exponential convergence to machine precision with $h\approx 0.1$. Most of our computations are done for the case where $p=3$ (Smagorinsky model), corresponding to   large eddy simulation \cite{S63}.  Though for periodic boundary conditions we present the analysis for a threshold value of $\mu$ only in the endpoint case  $p=11/5$, our numerical computations, show virtually no sensitivity to $p$ for two choices of $\mu$.  Finally we test an abridged nudging scheme which uses data only for the horizontal components of velocity.  We present evidence that synchronization still  holds for that scheme, though at a slower rate for the third component of velocity and pressure.

\subsection*{Organization of this paper} 
In section 2, we introduce the inequalities  and preliminary results used in the analysis.  Section 3 provides background on the Ladyzhenskaya model.   Later, in sections 4 and 5, we state and prove our main results, in which we give conditions under which the approximate solutions, obtained by the data assimilation algorithm, converge to the solution of  the Ladyzhenskaya  equations.  Numerical experiments, demonstrating and extending beyond the analytical results, are described in section 6. 

\section{Notation and Preliminaries}
 
 Let $\Omega \subset \mathbb{R}^d,  \, d = 2,3$, be  a bounded open Lipschitz  domain with  volume $|\Omega|$ and let $p \in [1 , \infty]$.  The
Lebesgue space $L^p(\Omega)$ is the space of all measurable functions $\bfv$ on $\Omega$ for which
$$\|\bfv\|_{L^p} :=\left( \int_{\Omega} |\bfv (\bfx)| ^p\, d\bfx \right)^{\frac{1}{p}} \, < \infty \hspace{1cm} \text{if} \hspace{0.2cm} p \in [1 , \infty),$$
$$\|\bfv\|_{L^{\infty}} := \esssup_{\bfx \in \Omega}  |\bfv (\bfx)| < \infty \hspace{1cm} \text{if} \hspace{0.2cm} p  = \infty. $$
The $ L^2$  norm and inner product will be denoted by $\|\cdot\|$ and $( \cdot ,  \cdot)$, respectively. The space $\dot{L}^2(\Omega)$ will consist of square integrable functions with zero spatial average.  
Let $\bfV$  be a Banach space of functions defined on $\Omega$ with the associated norm $\| \cdot \|_\bfV$. We denote by $L^p(a, b; \bfV)$,  $p \in [1 , \infty]$,  the  Bochner  space of  measurable \ functions $\bfv : (a, b) \rightarrow  \bfV$ such that
$$
\|\bfv\|_{L^p (a , b ; \bfV)} :=\left( \int_a^b \|\bfv(t)\|_\bfV ^p\, dt \right)^{\frac{1}{p}} \, < \infty \hspace{1cm} \text{if} \hspace{0.2cm} p \in [1 , \infty),$$
$$\|\bfv\|_{L^{\infty} (a , b ; \bfV)} := \esssup_{t  \in (a , b)} \|\bfv( t )\|_\bfV < \infty \hspace{1cm} \text{if} \hspace{0.2cm} p  = \infty. 
$$
The space  $W_0^{1,p}(\Omega)$ consists of all functions in
$W^{1,p}(\Omega)$ that vanish on the boundary $\partial \Omega$ (in the sense of traces)
$$ W_0^{1,p} = \big\{ \bfv:   \hspace{0.1cm}\bfv \in  W^{1,p}(\Omega)  \hspace{0.3cm} \text{and} \hspace{0.3cm}\bfv |_{\partial \Omega}=0  \big\}.$$
We introduce the Banach spaces of solenoidal functions
\begin{align*}
&L^2_\sigma (\Omega)= \lbrace \bfv:  \hspace{0.1cm}\bfv \in  L^2(\Omega)\text{,} \hspace{0.3cm} \nabla \cdot \bfv = 0  \hspace{0.3cm} \text{and} \hspace{0.3cm}  \bfv \cdot \bfn |_{\partial \Omega}=0
\rbrace,\\
&W_{\sigma}^{1,p}(\Omega) = \big\{ \bfv:   \hspace{0.1cm}\bfv \in  W^{1,p}(\Omega) \text{,} \hspace{0.3cm} \nabla \cdot \bfv = 0  \hspace{0.3cm} \text{and} \hspace{0.3cm} \bfv |_{\partial \Omega}=0  \big\}, 
\end{align*}
which are equipped with the same norms as $L^2(\Omega)$ and $W_0^{1,p}(\Omega)$, respectively.   We denote by  $\left( W_\sigma^{1,p}(\Omega) \right)'$ the dual space of $W_{\sigma}^{1,p}(\Omega)$.  We recall the following  inclusions for $p\geq2$
$$
W_{\sigma}^{1,p}(\Omega) \subset L_\sigma^2(\Omega) \subset  \left( W_\sigma^{1,p}(\Omega) \right)'  \hspace{0.3cm} \text{if} \hspace{0.3cm} \frac{2d}{d+2} \leq p < \infty,
$$
where these injections are continuous, dense and   compact. 
For matrix $A = (a_{ij})_{i,j=1}^3$, the Frobenius norm of the matrix $A$ is given by
$$
|A|_F = \left( \sum_{i,j =1}^3 (a_{ij})^2\right)^{\frac{1}{2}}=\left( A:A \right)^\frac12 . 
$$
The data assimilation method requires that the observational measurements $I_h(u)$ be given as linear interpolant observables satisfying $I_h : L^2(\Omega) \rightarrow L^2 (\Omega)$ such that
\begin{equation}\label{I_h}
\begin{split}
\|I_h \varphi \| & \leq c_I \|\varphi\|,  \hspace{1cm}   \forall\, \varphi \in  L^2(\Omega),  \\
 \|\varphi  - I_h\, \varphi\| & \leq c_0 \, h \| \varphi\|_{H^1(\Omega)}, \hspace{1cm}   \forall\, \varphi \in  H^1(\Omega).
\end{split}
\end{equation}
One example of such   interpolation operators includes projection onto Fourier modes with wave numbers $|k|\le 1/h$.   More physical examples are the volume elements and constant finite element interpolation \cite{Jones-Determining1992}.

\subsection*{Inequalities in Banach and Hilbert Spaces} We recall here some well-known  inequalities in Banach and Hilbert spaces  which can be found in the classical literature (see, e.g., \cites{A75, B10}).   Let $1 \leq p \leq\infty$,  we denote by $p'$ the conjugate exponent,
$\frac{1}{p} + \frac{1}{p'} = 1$.  Assume that $f \in L^p$ and $g \in L^{p'}$ with
$1\leq p \leq \infty$. Then
\begin{equation}\label{Holder}
 \tag{H\"older inequality}
\|fg\|_{L^1} \leq \|f\|_{L^p}\, \|g\|_{L^{p'}}.
\end{equation} 
Moreover, for  any  $a, b \geq 0$ and $\lambda > 0$ we have
\begin{equation}\label{Young}
 \tag{Young inequality}
a b \leq \lambda  \,a^p +  \left( p \, \lambda\right)^{- \frac{p'}{p}}\frac{ 1 }{p' } \, b^{p'}.
\end{equation}
Suppose $1 < p < \infty$, there exist two constants $c_{\text{P}}$ and $c_{\text{K}}$ such that for any $\bff \in W^{1, p}_0 (\Omega)$ 
\begin{equation}\label{Poincare}
 \tag{Poincar\'e inequality}
 \| \bff \|_{L^p} \leq c_{\text{P}} \| \nabla \bff \|_{L^p},
\end{equation}
and 
\begin{equation}\label{Korn}
 \tag{Korn inequality}
 \| \nabla \bff \|\leq \sqrt{2} \| \D \bff\|, \quad \text{if } p=2, \quad 
  \| \nabla \bff \|_{L^p}
 \leq c_{\text{K}} \| \D \bff \|_{L^p}, \quad \text{if } p\neq 2,
\end{equation}
where $ \D \bff= \frac12 \left[ \left( \nabla \bff \right) + \left( \nabla \bff \right)^T \right]$. The constants $c_{\text{P}}$ and $c_{\text{K}}$ depend only on $p$ and $\Omega$.
In the sequel, we will make use of the classical embedding theorems for Sobolev spaces
\begin{equation}\label{Embedding}
 \tag{Sobolev embedding}
 \begin{split}
&W^{1,2}(\Omega)  \hookrightarrow L^6 (\Omega),\\  
 &W^{1,3}(\Omega) \hookrightarrow  L^p (\Omega), \quad \forall \, p \in [1, \infty).
 \end{split}
\end{equation}
We recall the interpolation inequalities for Lebesgue and Sobolev spaces. 
Let $\bff \in L^p \cap L^q,$ with $1 \leq p , q \leq \infty$. Then, for all $r$ such that
$$\frac{1}{r} = \frac{\theta}{p} + \frac{1 - \theta}{q}, \hspace{0.5cm} 0 \leq \theta \leq1,$$
it follows that $\bff \in L^r$ and 
\begin{equation}\label{Interpolation}
\tag{Lebesgue interpolation inequality}
\| \bff\|_r \leq \| \bff\|^{\theta}_p \,  \| \bff\|^{1- \theta}_q.
\end{equation}
In addition, for any $\bff \in  H^1(\Omega)$, we have 
\begin{equation}\label{Ladyzhenskaya}
 \tag{Ladyzhenskaya’s inequality}
 \begin{split}
& \|\bff\|_{L^4} \leq C_L \|\bff\|^{\frac{1}{4}} \, \|\nabla \bff\|^{\frac{3}{4}}. 
 \end{split}
\end{equation}

We need the following fundamental results (see, e.g., \cites{A75,S87}).
\begin{thm}[\textbf{Schauder fixed-point theorem}] \label{Schauder fixed-point}
Let $\mathcal{X}$ be a Banach space, and let  $\mathcal{A}$ be a nonempty closed convex set in  $\mathcal{X}$.  Let $\mathcal{F}: \mathcal{A} \rightarrow \mathcal{A} $ be a continuous  map such that $\mathcal{F} ( \mathcal{A}) \subset \mathcal{K} $, where  $\mathcal{K}$ is a compact subset of $\mathcal{A}$. Then $\mathcal{F}$ has a fixed point in $\mathcal{K}$.
\end{thm}

\begin{thm}[\textbf{Aubin–Lions–Simon}] \label{Aubin-Lions-Simon} Let $\mathcal{B}_0 \subset \mathcal{B}_1 \subset \mathcal{B}_2 $ be three Banach
spaces. We assume that the embedding of $\mathcal{B}_1$  in $ \mathcal{B}_2$ is continuous and that the embedding of $ \mathcal{B}_0$ in $ \mathcal{B}_1$ is compact.  For $ 1 \leq p , r \leq +  \infty$  and $T>0$, we define
$$\mathtt{E}_{p,r} =  \left\{  v \in L^p\left(0,  T ;   \mathcal{B}_0 \right) \, , \, v_t \in  L^r\left( 0 , T ;  \mathcal{B}_2  \right) \right\}. $$
Then, we have
\begin{enumerate}
\item If $p < +  \infty$,  the embedding of $\mathtt{E}_{p,r}$  in  $ L^p\left(0,  T ;   \mathcal{B}_1 \right)$ is compact.
\item   If $p  = +  \infty$ and $r >1$,  the embedding of $\mathtt{E}_{p,r}$  in  $ \mathcal{C} \left(0,  T ;   \mathcal{B}_1 \right)$ is compact. 
\end{enumerate}
\end{thm}

Lastly, we report the following  Gronwall lemmas which will play a crucial role in our analysis (see, e.g., \cite{FMTT1983}). 
\begin{lem}[\textbf{Gronwall's lemma in differential form}]\label{Gronwall} Let $T \in \mathbb{R}^+$,  $f \in W^{1,1} (0, T)$ and $g,  \, \lambda  \in L^1 (0, T)$. Then
$$ f'(t) \leq \lambda (t)\, f(t) + g(t)
\hspace{1cm} \text{a.e.  in } [0, T]$$
implies for almost all $t \in [0 , T]$
$$ 
f(t) \leq f(0)\, \mathrm{e}^{\int_0^t\, \lambda(\tau)\, d\tau } + \int_0^t \, g(s) \, \mathrm{e}^{ \int_s^t\lambda(\tau)\, d \tau}\, ds.
$$
\end{lem}

\begin{lem} [\textbf{Uniform Gronwall lemma - 1}]
\label{unif-gronw}
Let $T \in \mathbb{R}^+$,  $f \in W^{1,1} (t_0, \infty)$ and $g,  \, \lambda  \in L^1_{\text{loc}} (t_0, \infty)$ which satisfy
$$ f'(t) \leq \lambda (t)\, f(t) + g(t) \hspace{1cm} \text{a.e.  in } (t_0, \infty),
$$
and
$$
\int_t^{t+r} \lambda(\tau) \, d \tau\leq a_1, 
\quad 
\int_t^{t+r} g(\tau) \, d \tau\leq a_2, 
\quad 
\int_t^{t+r} f(\tau) \, d \tau\leq a_3, 
\quad \forall \, t\geq t_0,
$$
for $r$, $a_1$, $a_2$ and $a_3$ positive.
Then, for $r>0$, we have
$$ 
f(t) \leq \left( \frac{a_3}{r}+ a_2\right) \mathrm{e}^{a_1}, \quad \forall \, t\geq t_0+r.
$$
\end{lem}

\begin{lem}[\textbf{Uniform Gronwall lemma - 2}]\label{Gronwall2} Let $T >0$ be fixed.  Suppose 
$$
\ddt Y + \alpha(t)\, Y(t) \leq 0
$$
where 
$$
\limsup_{t \rightarrow \infty}\, \int_t ^{t+T} \, \alpha(s)\, ds \geq \beta >0.
$$
Then $Y(t) \rightarrow 0$ exponentially as $t \rightarrow \infty$. 
\end{lem}

\section{The Ladyzhenskaya Model}
 The phenomenon that we consider in this section  is the motion of an  incompressible  viscous fluid in a bounded Lipschitz domain $\Omega  \subset \mathbb{R}^d,  \, d \in  \{2,3\}$ with no-slip  boundary conditions. Let $\bfu$  denote  the velocity field, $P$ the pressure, and $\bff$ the body force per unit mass.  In \cite{L67},  Ladyzhenskaya proposed the following mathematical model
\begin{equation}
 \label{Lady}
 \begin{split}
\partial_t \bfu+   (\bfu \cdot \nabla) \bfu   - \nabla \cdot  \mathbf{T}(\D \bfu) \ +  \nabla P  & = \bff,\\ 
\nabla \cdot \bfu  &= 0,\\
\bfu|_{\partial \Omega}& =0,
\end{split}
\end{equation}
where $\mathbf{T}$ denotes the  Cauchy stress of an incompressible and homogeneous fluid whose constitutive relation is given by 
\begin{equation}\label{Tensor}
\mathbf{T} (\D \bfu) =  2 \left(\nu_0  + \nu_1 |\D \bfu|_F^{p-2} \right)\, \D \bfu, \quad p\geq 2,
\end{equation}
with initial condition $\bfu(\cdot , 0) = \bfu_0(\cdot)$. Here, $\D \bfu= \frac12\left[ \left( \nabla \bfu \right) + \left( \nabla \bfu \right)^T \right] $, and $\nu_0$ and $\nu_1$ are positive parameters. It is worth mentioning that $\nu_0$ scales as $\frac{(\text{length})^2}{(\text{time})} $, and $\nu_1$ has dimension $(\text{time})^{p-3} \times (\text{length})^2$. In the literature, some works have been devoted to the case with $ \mathbf{T}=\mathbf{T}(\nabla \bfu)$, namely  $\D \bfu$ is replaced by the full velocity gradient $\nabla \bfu$ in \eqref{Tensor}. However, in such a case the model does not comply with the principle of frame indifference.

Before stating the well-posedness result,  we report the following property of the constitutive relation \eqref{Tensor}, which will be of key usefulness in the sequel.

\begin{prop}\label{propT-Mon}
Let $\mathbf{T}$ be as given in \eqref{Tensor}. For all  $\, A, B \in \mathbb{R}_{sym}^{3 \times 3}$,  we have
\begin{equation}
\label{T-mon}
\left( \mathbf{T}(A)-\mathbf{T}(B) \right): (A-B) \geq 2\, \nu_0 \, |A-B|_F^2.
\end{equation}
\end{prop}
The proof of the monotonicity property  \eqref{T-mon} can be found in  the above mentioned references. For the readers' convenience and in order to make the paper self-contained, we have included in the Appendix \ref{AppendixA} a short proof of the above property. We would like to stress the dependence of the factor $\nu_0$ in the lower bound, which will be exploited in the subsequent analysis.

 Taking advantage of the enhanced regularity due to \eqref{Tensor}, Ladyzhenskaya showed in \cites{L67, L68} that the weak solutions to \eqref{Lady} are global in time and unique for any Reynolds number and any exponent $p \geq \frac{5}{2}$.  For an overview, we refer the reader to  \cite{L69} and   \cite{L98}*{Theorem 7.2}, and to \cites{L94, L98,J01} for the existence of compact finite dimensional global attractor. 
 Later on, many contributions have been devoted to the analysis of the case 
 $1\leq p < \frac{5}{2}$. 
 Without any claim to give an exhaustive survey, we mention the existence of global measure-valued solutions for $\frac65 < p \leq \frac95$, 
 global weak solutions for $p > \frac95$ and global strong solutions for $p\geq \frac94 $ obtained in \cites{MNRR, MR1799487, MR1301173}.
 For the periodic case, enhanced results in terms of $p$ have been achieved as reported in \cites{MR1348587, MR2005}. In particular, the existence, but not uniqueness,  of global weak solutions fulfilling the energy equality holds for $p \geq \frac{11}{5}$. Moreover, under additional assumptions on the initial datum and the forcing term, global in time and unique strong solutions also exist.
 The asymptotic behavior in the same range of $p$ has been studied in \cites{MR1389407,malek2002large}.

We now state the well-posedness result for the model \eqref{Lady} proved in \cite{L67} (see also \cite{L92}). 


\begin{thm}[\textbf{Existence and uniqueness of  weak solutions}] 
\label{Lad-th}
Assume that $p \geq \frac{5}{2}$, $\bff \in L^{2} (0 , T ; L^2(\Omega))$ and  $\bfu_0 \in L_\sigma^2(\Omega)$. Problem \eqref{Lady} has a unique weak solution on $(0,\infty)$ satisfying for all $T>0$
$$ \bfu \in \mathcal{C}([0,T]; L_\sigma^2(\Omega)) \cap L^p (0 , T;  W_\sigma^{1,p}(\Omega)), \quad \partial_t \bfu \in L^{p'}(0,T;(W_\sigma^{1,p}(\Omega))'),$$
where $p'$ is the conjugate exponent of $p$, and 
$$
\lb \partial_t \bfu, \bfw \rb + ( (\bfu \cdot \nabla) \bfu, \bfw)   + ( \mathbf{T}(\D \bfu), \nabla \bfw)= (\bff, \bfw), \quad \forall \, \bfw \in W_\sigma^{1,p}(\Omega),
$$
for almost all $t\in [0,T]$. 
 Moreover,  the energy equality holds
\begin{equation}
\label{EnEq}
\frac12 \| \bfu(t)\|^2+ \int_{0}^t \left( 2\nu_0 \|\D \bfu (\tau)\|^2  + 2\nu_1 \|\D \bfu (\tau)\|_{L^p}^p \right) \, d\tau 
= \frac12 \| \bfu_0\|^2 + \int_0^t (\bff(\tau),\bfu(\tau)) \, d \tau, \quad \forall \, t \geq 0.
\end{equation}
\end{thm}

Let $\lambda_1$ be the smallest eigenvalue of the Stokes operator. Assume that $\bff$ is time independent. We denote by $G$ the Grashof number in three-dimensions defined as
\begin{equation}\label{Grashof}
G = \frac{\|\bff\|}{\nu_0^2 \,  \lambda_1^{3/4}}.
\end{equation}
We now give bounds on the solution $\bfu$ of \eqref{Lady} that will be used in our analysis. 

\begin{prop}\label{Prop1} Fix $T> 0$, and let $\bff \in L^2(\Omega)$.  Suppose that $\bfu$ is a weak solution of \eqref{Lady}, then we have
\begin{equation}
\label{u-decay}
\| \bfu (t)\|^2 \leq \| \bfu_0\|^2 \mathrm{e}^{- \lambda_1 \nu_0 t} +  \frac{\| \bff\|^2}{\lambda_1^2 \nu_0^2}  \left( 1- \mathrm{e}^{- \lambda_1 \nu_0 t}\right), \quad \forall \, t \geq 0.
\end{equation}
As a consequence, there exists a time $t_0 > 0$
such that for all $t \geq t_0$ we have
\begin{equation}\label{L2Bound}
\|\bfu(t)\|^2 \leq  2 \frac{\nu_0^2 G^2}{\lambda_1^\frac12}
\end{equation}
and 
\begin{equation}\label{W1pBound}
\int_t^{t+T}\, \left( \nu_0\|\D \bfu(\tau)\|^2+ \nu_1\| \D \bfu(\tau)\|_{L^p}^p\right) \, d\tau \leq 
2\left(1+ \nu_0 \lambda_1 T \right) 
\frac{\nu_0^2 G^2}{\lambda_1^\frac12}.
\end{equation}
\end{prop}
The proof of Proposition \ref{Prop1} is standard and thus omitted here. For the readers' convenience, we observe that \eqref{u-decay} follows from Lemma \ref{Gronwall} after dropping the $\nu_1$ term in \eqref{EnEq}. In addition, \eqref{W1pBound} is a consequence of \eqref{EnEq} and \eqref{L2Bound}.  

\section{The case \texorpdfstring{$p\geq\frac{5}{2}$}{p<=5/2} with no-slip boundary conditions}

In this section, we first analyze the nudging algorithm for the Ladyzhenskaya model with no-slip boundary conditions for $p\geq \frac{5}{2}$. After proving the global well-posedness of the week solution in Theorem \ref{thmE!}, 
we proceed to the task of finding conditions on $h$ and $\mu$ under which the approximate solution obtained by this algorithm converges to the reference solution over time, summarized in Theorem \ref{thm2}.

Let $I_h(\bfu(t))$ represent the observational measurements at a spatial resolution of size $h$ for $t >0$ satisfying \eqref{I_h}.  The approximating solution $\bfv$ with initial condition $\bfv(\cdot, 0) = \bfv_0(\cdot)$, chosen arbitrarily, shall be given by
\begin{equation} \label{Lady-DA}
\begin{split}
\partial_t \bfv+   (\bfv \cdot \nabla) \bfv   - \nabla \cdot  \mathbf{T}( \D \bfv) \ +  \nabla Q  = \bff - \mu I_h(\bfv - \bfu),&\\
\nabla \cdot \bfv  = 0,&\\
\bfv|_{\partial \Omega} =0.&
\end{split}
\end{equation}
The first result of this manuscript concerns the global well-posedness of weak solutions for the Data Assimilation algorithm.

\begin{thm}\label{thmE!}
Assume that $p \geq \frac{5}{2}$, $\bff\in L^2(0,T;L^2(\Omega))$ and $\bfv\in L_\sigma^2(\Omega)$.  Let $\bfu$ be the solution to problem \eqref{Lady} from Theorem \ref{Lad-th}.
The continuous data assimilation equations \eqref{Lady-DA} has a unique global weak solution that  satisfies for all $T>0$
$$ \bfv \in \mathcal{C}([0,T]; L^2_\sigma(\Omega)) \cap L^p (0 , T;  W_\sigma^{1,p}(\Omega))\cap W^{1,p'}(0,T; (W^{1,p}_\sigma(\Omega))')
$$
where $p'= \frac{p}{p-1}$, and 
\begin{equation}
\label{weak-DA}
\lb \partial_t \bfv, \bfw \rb + ( (\bfv \cdot \nabla) \bfv, \bfw)   + ( \mathbf{T}(\D \bfv), \nabla \bfw)= (\bff, \bfw)- \mu (I_h(\bfv-\bfu), \bfw), \quad \forall \, \bfw \in W_\sigma^{1,p}(\Omega),
\end{equation}
for almost all $t\in [0,T]$. 

\end{thm}

\begin{proof}
The  strategy  is to reformulate \eqref{Lady-DA} as a fixed-point problem. 
For any fixed $T>0$ , define 
\begin{equation}
\label{F-reg}
\mathcal{F}: L^2(0, T; L^2_\sigma (\Omega)) \rightarrow \mathcal{C} ([0, T]; L^2_\sigma (\Omega) ) \cap  L^p( 0, T ; W^{1,p}_\sigma (\Omega))\cap W^{1,p'}(0,T; (W^{1,p}_\sigma(\Omega))')
\end{equation}
by
$$\mathcal{F} (\bfw)= \bfv,$$
where $\bfv$ is a weak solution to the problem 
\begin{equation}
\label{fp-problem}
\begin{split} 
\partial_t \bfv+   (\bfv \cdot \nabla) \bfv   - \nabla \cdot  \mathbf{T}( \D \bfv) \ +  \nabla q  = \bff_\mu  - \mu I_h \bfw,&\\ 
\nabla \cdot \bfv = 0,&\\
\bfv|_{\partial \Omega} =0,\\
\bfv(\cdot,0)=\bfv_0(\cdot),
\end{split}
\end{equation}
for a given $\bfw \in  L^2 ( 0, T; L^2_\sigma (\Omega))$, 
with  $\bff_\mu = \bff + \mu \, I_h \bfu $. It is easy to verify that  $\bff_\mu \in  L^2 ( 0, T ; L^2(\Omega))$ since $I_h$ is a continuous and bounded linear operator.  The above map $\mathcal{F}$ is well-defined since the existence and uniqueness of a weak solution $\bfv$ for any given initial condition $\bfv_0 \in L_\sigma^2(\Omega)$ follows directly from Theorem \ref{Lad-th}. Now, define
\begin{equation}\label{A}
\mathcal{A} = \left\{ \bfw \in  L^2( 0, T; L^2_\sigma (\Omega))  \hspace{0.1cm} \text{:}   \hspace{0.1cm}  \int_0^t \, \|\bfw(\tau)\|^2  \, d \tau \leq c_1 e^{c_2 \, t}\, \,   \, \forall \, t \in [0, T]\right\},
\end{equation}
with 
$$
c_1 = 2 T \left( \|\bfv_0\|_{L^2_\sigma} + \int_0^{T}  \|\bff_\mu(\tau)\|\, d\tau \right)^2, \quad c_2 = 2 \, \mu^2\, c^2_I  \, T.
$$
To  apply the Schauder fixed-point theorem (see Theorem \ref{Schauder fixed-point}) to the above problem, we will verify the theorem's assumptions in the next five steps. 
\medskip

\subsection*{Step \RN{1}} 
We claim that $\mathcal{F} : \mathcal{A} \rightarrow \mathcal{A}$, i.e. $\mathcal{F} (\mathcal{A}) \subset \mathcal{A}.$

From the energy equality \eqref{EnEq}) and \eqref{I_h}, we have 
\begin{equation}
\begin{split}
\|\bfv(\tau)\| & \leq \|\bfv_0\| + \int_0^{\tau}  \|\bff_\mu(s)\|\, ds +  \mu  \int_0^{\tau}  \|I_h\bfw (s)\|\, ds\\
& \leq \|\bfv_0\| + \int_0^{\tau}  \|\bff_\mu(s)\|\, ds +  \mu\, c_I  \int_0^{\tau}  \|\bfw (s)\|\, ds
\end{split}
\end{equation}
for all $\tau \in [0,T]$.
By using Young's inequality and the H\"older's inequality, we obtain
\begin{equation}
\begin{split}
\|\bfv(\tau)\|^2 & \leq  2 \left( \|\bfv_0\| + \int_0^{\tau}  \|\bff_\mu(s)\|\, ds \right)^2 + 2 \, \mu^2\, c^2_I \left( \int_0^{\tau}  \|\bfw (s)\|\, ds\right)^2\\
&  \leq 2 \left( \|\bfv_0\| + \int_0^{\tau}  \|\bff_\mu(s)\|\, ds \right)^2 + 2 \, \mu^2\, c^2_I  \, T\, \left( \int_0^{\tau}  \|\bfw (s)\|^2\, ds\right).
\end{split}
\end{equation}
Since $\bfw \in \mathcal{A}$, we infer that

\begin{equation}
\begin{split}
\int_0^t  \|\bfv(\tau)\|^2 \, d\tau 
&\leq  2 \int_0^t \left( \|\bfv_0\| + \int_0^{T}  \|\bff_\mu(s)\|\, ds \right)^2 d \tau + 2 \, \mu^2\, c^2_I  \, T\, \int_0^t \int_0^{\tau}  \|\bfw (s)\|^2\, ds \, d \tau \\
&\leq c_1 +  \, c_2  \int_0^t \, \left(c_1\, e^{c_2 \tau}\right) \, d\tau = c_1 e^{c_2\, t},
\end{split}
\end{equation}
which, in turn, entails $\bfu \in \mathcal{A}.$

\subsection*{Step \RN{2}} $\mathcal{A} $ is a closed set in   $L^2 \left( 0, T; L^2_\sigma (\Omega) \right)$.

Assume that $\{\bfw_n\}_{n=0}^\infty \subset \mathcal{A}$ is such that $\bfw_n \rightarrow \bfw$ in   $L^2 \left( 0, T; L^2_\sigma (\Omega) \right)$.   It follows that $\mathcal{A}$ is closed from the following argument
\begin{equation*}
\begin{split}
& \int_0^t \|\bfw(\tau)\|^2\, d\tau = \lim_{n \rightarrow \infty}\int_0^t \|\bfw_n(\tau)\|^2\, d\tau   \leq \lim_{n \rightarrow \infty} c_1 e^{c_2 t} = c_1 e^{c_2 t}, \quad \forall \, t \in [0,T].\\
\end{split}
\end{equation*}

\subsection*{Step \RN{3}} $\mathcal{A}$ is convex set in $ L^2 \left( 0, T; L^2_\sigma (\Omega) \right) $.

Let $\bfw_1, \bfw_2 \in \mathcal{A}$,  then $\lambda \bfw_1 + (1 - \lambda) \bfw_2 \in L^2 \left( 0, T; L^2_\sigma (\Omega) \right)$ for any $\lambda \in [0 , 1]$.
We compute
\begin{equation*}
\begin{split}
&\int_0^t \| \lambda \bfw_1(\tau) + (1 - \lambda) \bfw_2(\tau) \|^2\, d \tau\\  
&= \, \lambda^2 \int_0^t \|  \bfw_1(\tau) \|^2\, d \tau 
+ 2 \lambda (1-\lambda) \int_0^t (\bfw_1(\tau),\bfw_2(\tau)) \, d \tau 
+   \, (1 -\lambda)^2 \int_0^t \|  \bfw_2 (\tau) \|^2\, d \tau\\
& \leq \, \lambda^2 \int_0^t \|  \bfw_1(\tau) \|^2\, d \tau 
+ 2 \lambda (1-\lambda) \left( \int_0^t \|\bfw_1(\tau)\|^2 \, d \tau \right)^\frac12  \left( \int_0^t \|\bfw_2(\tau))\|^2 \, d \tau \right)^\frac12\\ 
& \quad \, +   \, (1 -\lambda)^2 \int_0^t \|  \bfw_2 (\tau) \|^2\, d \tau\\
& \leq  \left( \lambda^2 + (1 - \lambda)^2 + 2\lambda (1-\lambda) \right) \, c_1 e^{c_2 t} \\
&\leq c_1 e^{c_2 t},
\end{split}
\end{equation*}
which means $\lambda \bfw_1 + (1 - \lambda) \bfw_2 \in \mathcal{A}$, proving the convexity.

\subsection*{Step \RN{4}} $\mathcal{F}: \mathcal{A} \rightarrow \mathcal{A}$ is continuous.

Consider $\{\bfw_n \}_{n=1}^\infty \subset \mathcal{A}$ such that $\bfw_n \rightarrow \bfw$ in $L^2\left( 0 , T; L^2_\sigma(\Omega)\right)$. We are required to show that $\bfv_n = \mathcal{F} (\bfw_n) \rightarrow \mathcal{F} (\bfw) = \bfv$ in $L^2\left( 0 , T; L^2_\sigma(\Omega)\right)$. First, 
define the difference $\psi_n = \bfv_n - \bfv$, which solves 
\begin{equation*}
\begin{split}
\lb \partial_t \psi_n, \varphi \rb + \left( \bfv_n \cdot \nabla \bfv_n, \varphi\right) - \left( \bfv \cdot \nabla \bfv,  \varphi\right)  + \left(  \mathbf{T} (\D \bfv_n) - \mathbf{T} (\D \bfv) , \nabla \varphi\right) = - \mu \left( I_h (\bfw_n - \bfw) , \varphi\right)
\end{split}
\end{equation*}
for all $\varphi \in W_\sigma^{1,3}(\Omega)$, for almost all $t \in [0, T]$.  Thanks to \cite{MNRR}*{Lemma 2.45}, the incompressibility condition and the regularity \eqref{F-reg}, choosing $\varphi = \psi_n$ in the above equation, we obtain 
\begin{equation*}
\begin{split}
\frac{1}{2} \ddt \|\psi_n\|^2 + \left( \psi_n \cdot \nabla \bfv , \psi_n\right) +  \left(  \mathbf{T} (\D \bfv_n) - \mathbf{T} (\D \bfv) , \D \bfv_n - \D \bfv\right) = - \mu \left( I_h (\bfw_n - \bfw) , \psi_n\right).
\end{split}
\end{equation*}
By exploiting \eqref{T-mon}, the Korn inequality, the \ref{Holder} with $p' = \frac{p}{p-1}$  and the \ref{Interpolation} in $L^p$-spaces with $\theta = 1 - \frac{3}{2p}$,  we find
\begin{equation*}
\begin{split}
\frac{1}{2} \ddt \|\psi_n\|^2  + \nu_0 \|\nabla \psi_n\|^2 & \leq \left|  \left( \psi_n \cdot \nabla \bfv , \psi_n\right)\right| + \left| \mu \left( I_h (\bfw_n - \bfw) , \psi_n\right)\right|\\
& \leq   \|\psi_n\|_{L^{2p'}}^2  \|\nabla \bfv\|_{L^p} + \mu \, \|I_h (\bfw_n - \bfw)\| \,  \|\psi_n\|\\
& \leq \|\psi_n\|^{2 - \frac{3}{p}} \, \|\psi_n\|_{L^6}^{\frac{3}{p}} \, \|\nabla \bfv\|_{L^p} + \mu \, c_I\, \|\bfw_n - \bfw\| \,  \|\psi_n\|\\
& \leq  c_{S}\|\psi_n\|^{2 - \frac{3}{p}} \, \|\nabla \psi_n\|^{\frac{3}{p}} \, \|\nabla \bfv\|_{L^p} + \mu \, c_I\, \|\bfw_n - \bfw\| \,  \|\psi_n\|\\
& \leq \frac{\nu_0}{2}  \|\nabla \psi_n\|^2 + \left( \tilde{c}\, \nu_0^{-\frac{3}{2p-3}} \, c_S^\frac{2p}{2p-3} \, \|\nabla \bfv\|_{L^p}^{\frac{2p}{2p-3}} + \frac14 \right) \|\psi_n\|^2 +  \mu^2 \, c^2_I\, \|\bfw_n - \bfw\|^2,\\
\end{split}
\end{equation*}
where $\tilde{c}$ only depends on $p$. In the above estimate, the constant $c_S$ denotes the \ref{Embedding} $H_0^1(\Omega) \hookrightarrow L^6 (\Omega)$. Therefore,  we obtain
\begin{equation*}
 \ddt \|\psi_n\|^2 \leq  \left( \frac{1}{4} + \tilde{c}\, \nu_0^{-\frac{3}{2p-3}} \, c_S^\frac{2p}{2p-3} \,\|\nabla \bfv\|_{L^p}^{\frac{2p}{2p-3}}  \right)  \|\psi_n\|^2 +  \mu^2\, c_I^2 \, \|\bfw_n - \bfw\|^2 .
\end{equation*}
Applying the Gronwall lemma (see Lemma \ref{Gronwall}) to the above inequality, we get
\begin{equation*}
 \|\psi_n(t)\|^2 \leq  \|\psi_n (0)\|^2\, \mathrm{e}^{\int_0^t\, \lambda(\tau)\, d\tau} + \mu^2\, c_I^2 \int_0^t \,  \|\bfw_n(s) - \bfw (s)\|^2 \, \mathrm{e}^{\int_s^t\lambda(\tau)\, d \tau}\, ds, 
\end{equation*}
for all $t \in [0,T]$,
where $$\lambda(\tau) =   \frac{1}{4} + \tilde{c}\, \nu_0^{-\frac{3}{2p-3}} \, c_S^\frac{2p}{2p-3} \,\|\nabla \bfv(\tau)\|_{L^p}^{\frac{2p}{2p-3}} .$$
Note that having $p \geq \frac{5}{2}$ yields $\frac{2p}{2p-3} \leq p$, thereby the regularity $\bfv \in L^p (0 , T;  W_\sigma^{1,p}(\Omega)) $  entails that $\lambda(\tau)  \in L^1[0, T]$.  In light of $\psi_n(0) =0$, we are led to
\begin{equation*}
 \|\psi_n\|_{L^\infty(0,T;L_\sigma^2(\Omega))} \leq  \mu\, c_I \, e^{\frac12\|\lambda\|_{L^1(0,T)}}\,  \|\bfw_n - \bfw\|_{L^2\left( 0 , T; L^2_\sigma(\Omega)\right)}.
\end{equation*}
Since the right-hand side converges to $0$ as $n\rightarrow \infty$, this implies the continuity of $\mathcal{F}$.

\subsection*{Step \RN{5}} We construct a compact subset $\,\mathcal{K}$ of $\mathcal{A}$ such that $\mathcal{F} (\mathcal{A}) \subset \mathcal{K}$.
From the energy equality \eqref{EnEq} written for the solution to \eqref{fp-problem}, and after using the \ref{Holder}, the Korn inequality and \eqref{I_h},  we have
\begin{equation*}
\begin{split}
 \| \bfv(t)\|^2   &  
 + \int_0^t\,   \left(2\nu_0 \|\nabla \bfv(\tau)\|^2   + \frac{4\nu_1}{c_{\text{K}}^p} \|\nabla \bfv(\tau)\|_{L^p}^p \right)\, d\tau \leq \| \bfv_0\|^2 +  2\int_0^t \, \left( \|\bff_\mu(\tau)\|  + \mu \|I_h \bfw (\tau)\|\right) \| \bfv(\tau)\| \, d\tau \\
 & \leq  \| \bfv_0\|^2 + \frac{2}{\lambda_1} \int_0^t \, \left( \|\bff_\mu(\tau)\|  + \mu \|I_h \bfw (\tau)\|\right) \| \nabla \bfv(\tau)\| \, d\tau\\
 & \leq \| \bfv_0\|^2 +  \nu_0 \int_{0}^t \, \| \nabla \bfv(\tau)\|^2 \, d \tau + \frac{1}{2 \nu_0 \lambda_1^2}
\int_0^t \, \left(\|\bff_\mu(\tau)\|  + \mu \|I_h \bfw (\tau)\|\right)^2 \, d \tau
\\
& \leq \| \bfv_0\|^2 +  \nu_0 \int_{0}^t \, \| \nabla \bfv(\tau)\|^2 \, d \tau + \frac{1}{\nu_0 \lambda_1^2}
\int_0^t \,  \left(\|\bff_\mu(\tau)\|^2  \, +\,  \mu^2 c_{\text{I}}^2 
\|\bfw (\tau)\|^2\right) \, d \tau,
\end{split}
\end{equation*}
 for all $t\in [0,T]$. Thus, we arrive at 
\begin{equation*}
\| \bfv(t)\|^2  + \int_0^t\,   \left( \nu_0 \|\nabla \bfv(\tau)\|^2   + \frac{4\nu_1}{c_{\text{K}}^p} \|\nabla \bfv(\tau)\|_{L^p}^p \right)\, d\tau \leq
\| \bfv_0\|^2+ \frac{1}{\nu_0 \lambda_1^2} \| \bff_\mu\|^2_{L^2 \left( 0, T; L^2_\sigma (\Omega) \right)} \, + \, 
\frac{\mu^2 c_{\text{I}}^2}{\nu_0 \lambda_1^2} \, c_1 e^{c_2 T} \coloneqq \tilde{c}_0,
\end{equation*}
 With  $\tilde{c}_0$ defined as above,  we deduce that 
\begin{equation}
\label{c1tilde}
 \| \bfv\|_{{L^\infty \left( 0, T; L^2_\sigma (\Omega) \right)}} \leq \sqrt{ \tilde{c}_0} \coloneqq  \tilde{c}_1, \quad 
  \| \bfv\|_{{L^p \left( 0, T; W^{1,p}_\sigma (\Omega) \right)}} \leq \left(\frac{\tilde{c}_0 c_{\text{K}}^p}{2 \,  \nu_1} \right)^\frac{1}{p} \coloneqq  \tilde{c}_2.
\end{equation}
Then, we infer that 
$$\mathcal{F} (\mathcal{A}) \subset \mathcal{B} = \left\{  \bfv \in \mathcal{A} \, : \| \bfv\|_{{L^\infty \left( 0, T; L^2_\sigma (\Omega) \right)}} \leq \tilde{c}_1 \hspace{0.4cm}\text{and} \hspace{0.4cm}  \| \bfv\|_{{L^p \left( 0, T; W^{1,p}_\sigma (\Omega) \right)}} \leq \tilde{c}_2 \right\}.$$

Next we investigate the time derivative $\partial_t \bfv$. We recall the weak formulation of \eqref{fp-problem} 
\begin{equation*}
\begin{split}
\lb \partial_t \bfv, \varphi \rb+ \int_{\Omega}  \bfv \cdot \nabla \bfv \cdot\, \varphi\, d\bfx + \int_{\Omega}  2\nu_0 \, \D \bfv : \D \varphi\, &+ \, 2\nu_1 \, |\D \bfv|^{p-2}_F \,  \D \bfv : \D \varphi \, d\bfx \\
& =  \int_{\Omega}  \bff_\mu \cdot \varphi \, d\bfx - \mu \int_{\Omega}  I_h \bfw\cdot \varphi \,d\bfx
\end{split}
\end{equation*}
for all $\varphi \in W_\sigma^{1,p}(\Omega)$, for almost all $t \in [0,T]$. Due to the incompressiblity condition, the nonlinear term can be written as $ \bfv \cdot \nabla \bfv = \nabla \cdot \left( \bfv \otimes  \bfv\right)$. Then, we have
\begin{equation*}
\begin{split}
\left| \lb \partial_t \bfv ,  \varphi\, \rb  \right|  \leq  \left| \int_{\Omega}  \bfv   \otimes \bfv : \nabla \varphi\, d\bfx \right | & + \left| \int_{\Omega} 2 \nu_0 \, \D \bfv : \nabla \varphi\,  + \, 2\nu_1 \, |\D \bfv|^{p-2}_F\,\D \bfv : \nabla \varphi \, d\bfx \right|\\
&+ \left|  \int_{\Omega}  \bff_\mu \cdot \varphi \, d\bfx \right| +  \mu  \left|\int_{\Omega}  I_h \bfw\cdot \varphi \,d\bfx \right|.
\end{split}
\end{equation*}
Let $p' = \frac{p}{p-1}$, and note that $p' < p$ for $p\geq \frac{5}{2}$.   Using the \ref{Holder} along with \eqref{I_h} yields
\begin{equation*}
\begin{split}
\left| \lb \partial_t \bfv ,  \varphi\, \rb  \right|  \leq \|\bfv\|_{L^{2p'}}^2\, \|  \nabla \varphi\|_{L^p} + 2\nu_0 \|  \D \bfv\|_{L^{p'}}\, \| \nabla \varphi \|_{L^p} + 2\nu_1 \|  \D \bfv\|^{p-1}_{L^p} \, \|  \nabla \varphi\|_{L^p} + \| \bff_\mu \|  \, \| \varphi \|  +  \mu c_I\,  \| \bfw \|  \, \| \varphi \|.
\end{split}
\end{equation*}
By taking supremum of the above inequality over all  $\varphi \in W_\sigma^{1,p}(\Omega)$   such that  $\| \varphi\|_{W_\sigma^{1,p}(\Omega)} = 1$,  and using the \ref{Interpolation}, we obtain
\begin{equation*}
\begin{split}
\|\partial_t \bfv\|_{\left(W_\sigma^{1,p}(\Omega)\right)'} &
\leq  \|\bfv\|^2_{L^{2p'}} \, + 2\nu_0 \|  \D \bfv\|_{L^{p'}}\,  + 2\nu_1 \|  \D \bfv\|^{p-1}_{L^p} \, + C \, \| \bff_\mu \|   + \mu \,  c_I \, C \| \bfw\| \\
& \leq  C\,  \|\bfv\|^{\frac{2p-3}{p}}\,  \| \nabla \bfv\|^{\frac{3}{p}}+ 
\nu_0 \, C \,  \| \nabla \bfv\|_{L^2}  + \nu_1 \, C \|  \nabla \bfv\|^{p-1}_{L^p} \, + C \, \| \bff_\mu \|   + \mu \,  c_I \, C \, \| \bfw\|\\
&  \leq  C \, \|\bfv\|^{\frac{2p-3}{p}}\,  \| \nabla \bfv\|_{L^p}^{\frac{3}{p}} + \nu_0 \, C \,  \| \nabla \bfv\|_{L^p}  + \nu_1 \, C \, \|  \nabla \bfv\|^{p-1}_{L^p} \, + C \, \| \bff_\mu \|   + \mu \,  c_I \, C \, \| \bfw\|,
\end{split}
\end{equation*}
where $C$ only depends on $p$ and $\Omega$.  Hence, 
\begin{equation*}
\begin{split}
& \|\partial_t \bfv\|^{p'}_ {L^{p'} \left( 0 , T ; \,  \left(W_\sigma^{1,p}(\Omega)\right)' \right)} = \int_0^T \|\partial_t \bfv(\tau)\|^{p'}_{\left(W_\sigma^{1,p}(\Omega) \right)'}\, d \tau \\
 & \leq C \int_0^T  \|\bfv(\tau)\|^{\frac{2p-3}{p-1}}\,  \|\nabla\bfv(\tau)\|_{L^p}^{\frac{3}{p-1}} \, d\tau  + \nu_0^{p'} \, C \int_0^T \|\nabla \bfv (\tau)\|_{L^p}^{\frac{p}{p-1}}\, d \tau + \nu_1^{p'} C \int_0^T\|\nabla \bfv(\tau)\|_{L^p}^{p} \, d \tau\\
 & \quad + \, C \int_0^T \, \| \bff_\mu(\tau) \|^{p'}\, d\tau + \mu^{p'} \,  c_I^{p'} \, C \int_0^T \| \bfw(\tau)\|^{p'} \,dt\\
 &\leq C \, \|\bfv\|^{\frac{2p-3}{p-1}}_{L^\infty \left(0, T; L^2_\sigma(\Omega)\right)} \, T^{\frac{1}{\alpha}} \,  \|\bfv\|_{L^p\left(0, T; W^{1,p}_\sigma(\Omega)\right)}^{\frac{3}{p-1}}  \, + \,  \nu_0^{p'} \, C\, T^{\frac{1}{\beta}} \, \|\bfv\|_{L^p\left(0, T; W^{1,p}_\sigma(\Omega)\right)}^{p'} \, + \, \nu_1^{p'} \,C\, \|\bfv\|_{L^p\left(0, T; W^{1,p}_\sigma(\Omega)\right)}^p\\
 & \quad + \,  C \, T^{\frac{1}{\gamma}} \, \|\bff_\mu\|^{p'}_{L^2\left(0, T; L^2_\sigma(\Omega)\right)}  + \mu^{p'} \,  c_I^{p'} \, C \,  \, T^{\frac{1}{\gamma}} \, \|\bfw\|^{p'}_{L^2\left(0, T; L^2_\sigma(\Omega)\right)} \\
& \coloneqq \tilde{c}_3^{p'}.
\end{split}
\end{equation*}
where $\alpha,  \beta $ and $\gamma$ are the conjugate exponents to $(p-1)p/3, p-1 $  and $2/p'$, respectively, and the constant $C$ depends only on $p$ and $\Omega$. Given $\tilde{c}_3^{p'}$ as above, we have
\begin{equation}
\label{c3tilde}
\|\partial_t \bfv\|_ {L^{p'} \left( 0 , T ; \,  \left(W_\sigma^{1,p}(\Omega)\right)' \right)} \leq \tilde{c}_3.
\end{equation}

Finally, with $\tilde{c}_1, \tilde{c}_2, \tilde{c}_3$ given in \eqref{c1tilde} and \eqref{c3tilde}, respectively, we infer that
\begin{equation}\label{K}
\mathcal{F} (\mathcal{A}) \subset \mathcal{K},
\end{equation}
where 
$$
\mathcal{K}= \left\{  \bfv \in \mathcal{A} \, : \| \bfv\|_{{L^\infty \left( 0, T; L^2_\sigma (\Omega) \right)}} \leq \tilde{c}_1, \, \, \| \bfv\|_{{L^p \left( 0, T; W^{1,p}_\sigma (\Omega) \right)}} \leq \tilde{c}_2  \hspace{0.2cm}\text{and} \hspace{0.2cm} \|\partial_t \bfv\|_ {L^{p'} \left( 0 , T ; \,  \left(W_\sigma^{1,p}(\Omega)\right)' \right)} \leq \tilde{c}_3\right\}.  
$$
We are left to show that $\mathcal{K}$ is a compact subset of $\mathcal{A}$. Since  $W_{\sigma}^{1,p}(\Omega) \subset L_\sigma^2(\Omega) \subset  \big( W_\sigma^{1,p}(\Omega) \big)'$, 
thanks to Theorem  \ref{Aubin-Lions-Simon}, we deduce that
$\mathcal{K}$ is compactly embedded in $ L^p\left( 0 , T, L^2_\sigma(\Omega)\right)$,
and, in turn, in $L^2\left( 0 , T, L^2_\sigma(\Omega)\right)$  since $p \geq \frac{5}{2}$.   Therefore, to summarize it is proved that 
$$\mathcal{F}(\mathcal{A}) \subset \mathcal{K} \overset{c}{\hookrightarrow} \mathcal{A}$$
where  $\mathcal{K}$ is a compact subset of $\mathcal{A}$ with respect to the norm $L^2\left( 0 , T, L^2_\sigma(\Omega)\right)$. 
As a consequence of Theorem \ref{Schauder fixed-point}, $\mathcal{F}: \mathcal{A} \rightarrow \mathcal{A}$ has a fixed point in $\mathcal{K}$, which implies the  existence result in Theorem \ref{thmE!}. Lastly, the uniqueness of the weak solution to problem \eqref{Lady-DA} is obtained from the same  argument of Step \RN{4} by replacing $\bfv_n$ and $\bfv$ with two solutions $\bfv_1$ and $\bfv_2$, respectively, originating from the same initial datum $\bfv_0$.  
\end{proof}
\medskip
Next, we prove the convergence result.

\begin{thm}\label{thm2}
  For $p \geq \frac{5}{2}$, let $\bff\in L^2(\Omega)$ and let $\bfu$ be a weak solution of \eqref{Lady} with no-slip Dirichlet boundary conditions departing from $\bfu_0 \in L_\sigma^2(\Omega)$. Let $\bfv$ be the solution to the data assimilation algorithm given by \eqref{Lady-DA}. Then, for $\mu$ large enough such that 
  $$\mu \geq \tilde{c}
   \, \nu_0^{\frac{3}{2p-3}} \, \nu_1^{\frac{-2}{2p-3}} \lambda_1^{\frac{1}{2p-3}}\, G^{\frac{4}{2p-3}},$$
   and $h$ small enough such that 
 $$\mu \,  c_0^2\,  h^2 \leq \nu_0,$$
 where $\tilde{c}$  is a dimensionless number depending only on $p$ and $\Omega$, while $c_0$ is dimensionless constant given in \eqref{I_h},  we have 
$$
\| \bfu(t) - \bfv(t)\|_{L^2(\Omega)} \rightarrow 0,
$$
at an exponential rate, as $t \rightarrow \infty$. 
\end{thm}

\begin{proof}
Subtracting \eqref{Lady-DA} and \eqref{Lady},  the difference $\bfe = \bfu - \bfv$ satisfies the following error equation 
\begin{align} 
\label{thm2Eq1}
\lb \partial_t \bfe, \bfw \rb+
((\bfu \cdot \nabla) \bfu, \bfw) 
- ((\bfv \cdot \nabla) \bfv, \bfw) 
+ (\mathbf{T} ( \D\bfu) -  \mathbf{T} ( \D\bfv), \D \bfw) 
= - \mu\, ( I_h\bfe, \bfw).
\end{align}
Since
$$ 
(\bfu \cdot \nabla) \bfu   - (\bfv \cdot \nabla) \bfv  =  (\bfe \cdot \nabla) \bfu + (\bfv \cdot \nabla) \bfe,
$$
taking $\bfw  = \bfe$ and using \cite{MNRR}*{Lemma 2.45}, the Korn inequality and \eqref{T-mon}, we obtain
 \begin{equation} \label{thm2Eq2}
\frac{1}{2}  \ddt \|\bfe\|^2 + 
\nu_0 \|\nabla \bfe\|^2 \leq  - ( (\bfe \cdot \nabla) \bfu\, , \, \bfe) - (\mu\,  I_h\bfe\, , \, \bfe). 
\end{equation}
In light of \eqref{I_h} and the assumption $\mu \,  c_0^2\,  h^2 \leq \nu_0$, one can estimate   the nudging term in (\ref{thm2Eq2})  as 
\begin{equation}\label{thm2Eq4}
\begin{split}
- \mu\, (\,  I_h\bfe , \bfe) & =  - \mu\, ( I_h\bfe - \bfe + \bfe , \bfe)\\
& =  \mu\, (\bfe  - I_h\bfe , \bfe) - \mu \|\bfe \|^2\\
& \leq \frac{\mu}{2} \|\bfe  - I_h\bfe \|^2 + \frac{\mu}{2} \|\bfe\|^2  - \mu \|\bfe\|^2\\
& \leq\frac{\mu}{2} c_0^2\,  h^2 \|\nabla \bfe\|^2 -  \frac{\mu}{2} \|\bfe\|^2 \\
& \leq \frac{\nu_0}{2} \|\nabla \bfe\|^2  -  \frac{\mu}{2} \|\bfe\|^2.
\end{split}
\end{equation}
Thus,  we have
 \begin{equation} \label{thm2Eq5}
\frac{1}{2}  \ddt \|\bfe\|^2 +  \frac{\nu_0}{2} \|\nabla \bfe\|^2  \leq  | \left( \left(\bfe \cdot \nabla\right) \bfu , \bfe\right)  |-  \frac{\mu}{2} \|\bfe\|^2.
\end{equation}
Take $p$ and $p'$ to be conjugate numbers, i.e., $ p' = \frac{p}{p-1}$, and  apply the \ref{Interpolation},  \ref{Embedding} and \ref{Young}  to estimate the above nonlinear term as
\begin{equation}\label{thm2Eq6}
\begin{split}
 | \left( \left(\bfe \cdot \nabla\right) \bfu , \bfe\right) & | \leq   \| \bfe^2\|_{L^{p'}}  \, \|\nabla \bfu\|_{L^p} =     \| \bfe\|^2_{L^{2p'}}\, \|\nabla \bfu\|_{L^p}  \leq  \|\bfe\|^{2 - \frac{3}{p}}\, \|\bfe\|_{L^6}^{\frac{3}{p}} \,  \|\nabla \bfu\|_{L^p} \\
 & \leq  c_S^\frac{3}{p} \|\bfe\|^{2 - \frac{3}{p}}\, \| \nabla \bfe\|^{\frac{3}{p}} \, \|\nabla \bfu\|_{L^p}  \leq \frac{\nu_0}{2}  \| \nabla \bfe\|^2 + \frac{\bar{c}}{2}  \, \nu_0^{\frac{3}{3-2p}}\,  \|\nabla \bfu\|_{L^p}^{\frac{2p}{2p-3}} \, \|\bfe\|^2,
\end{split}
\end{equation}
for some $\bar{c}$ depending only on $p$ and $\Omega$. 
Inserting \eqref{thm2Eq6} in \eqref{thm2Eq5}, we get 
\begin{equation}\label{thm2Eq7}
\ddt \|\bfe\|^2 + \big( \mu -  \bar{c} \,  \nu_0^{\frac{3}{3-2p}} \,  \|\nabla \bfu\|_{L^p}^{\frac{2p}{2p-3}}\big)\,  \|\bfe\|^2 \leq 0.
\end{equation}
With Lemma \ref{Gronwall2} in mind, denote 
$$\alpha(t) =   \mu -  \bar{c} \,  \nu_0^{\frac{3}{3-2p}} \,  \|\nabla \bfu(t)\|_{L^p}^{\frac{2p}{2p-3}}.$$
Applying H\"older's inequality, and choosing $T =\,  \left( \nu_0\, \lambda_1\right)^{-1}$ in \eqref{W1pBound},  we obtain for $p\geq \frac52$
\begin{equation*}
\begin{split}
\int_t^{t+T}\, \alpha(s)\, ds & = \mu T  - \bar{c} \,  \nu_0^{\frac{3}{3-2p}} \, \int_t^{t+T}\,  \|\nabla \bfu (s)\|_{L^p}^{\frac{2p}{2p-3}} ds\\
& \geq \mu T  - \bar{c} \,  \nu_0^{\frac{3}{3-2p}}\,  T^{\frac{2p-5}{2p-3}}\,  \left(\int_t^{t+T} \|\nabla \bfu (s)\|_{L^p}^{p} \, ds\right)^{\frac{2}{2p-3}}\\
& \geq \mu T  - \bar{c} \,  \nu_0^{\frac{3}{3-2p}}\,  T^{\frac{2p-5}{2p-3}} \left( 2 c_K^p \left( 1 + \nu_0\, \lambda_1 \, T\right)\, \nu_0^2 \, \nu_1^{-1}\, \lambda_1^{-\frac{1}{2}}\, G^2 \right)^{\frac{2}{2p-3}}\\
& = \frac{\mu}{\nu_0\, \lambda_1} - 2^\frac{4}{2p-3} \bar{c} \, 
c_K^\frac{2p}{2p-3}\, \nu_0^{\frac{6-2p}{2p -3}}\,  \nu_1^{ \frac{-2}{2p-3}}\, \lambda_1^{\frac{4-2p}{2p-3}} \, G^{\frac{4}{2p-3}}.
\end{split}
\end{equation*}
Thus, from above and with $\mu \geq 2^{1+\frac{4}{2p-3}} \,  \bar{c} \, c_K^\frac{2p}{2p-3}\, \nu_0^{\frac{3}{2p-3}} \, \nu_1^{\frac{-2}{2p-3}}\,  \lambda_1^{\frac{1}{2p-3}}\, G^{\frac{4}{2p-3}}$ we have 
$$ \int_t^{t+T}\, \alpha(s)\, ds \geq 2^\frac{4}{2p-3} \, \bar{c} \, c_K^\frac{2p}{2p-3}\, \nu_0^{\frac{6 - 2p}{2p -3}}\,  \nu_1^{ \frac{-2}{2p-3}}\, \left( \frac{1}{\lambda_1} \right)^{\frac{2p-4}{2p-3}} \, G^{\frac{4}{2p-3}} >0,$$
and finally by applying  Lemma \ref{Gronwall2} to \eqref{thm2Eq7},  we conclude that $\|\bfe\| = \| \bfu - \bfv\| \rightarrow 0$ exponentially fast as $t \rightarrow \infty$.
\end{proof}

\section{The case \texorpdfstring{$p=\frac{11}{5}$}{p=11/5} with periodic boundary conditions}

In this section we study the dynamics of strong solutions for the Ladyzhenskaya model \eqref{Lady}$_{1-2}$ and the corresponding data assimilation algorithm \eqref{Lady-DA}$_{1-2}$ in $\Omega=[0,2\pi]^3$ completed with periodic boundary conditions. 

Since the average velocity $\overline{\bfu}(t)=\int_{\Omega} \bfu (\bfx,t) \, d \bfx$ is an invariant of the flow provided that $\int_{\Omega} \bff (\bfx,t) \, d \bfx=0$ and the interpolant operators (volume elements or Fourier modes) have zero spatial average, we consider without loss of generality that $\overline{\bfv}(t)=0$ for all $t\geq 0$. 

\begin{thm}[\textbf{Existence of weak solutions and their propagation of regularity}]
\label{115}
Let $p= \frac{11}{5}$, $\bff\in L^2(0,T;\dot{L}^2(\Omega))$ and $\bfu_0 \in \dot{L}^2_\sigma$. Then, there exists a weak solution $\bfu$ to \eqref{Lady}$_{1-2}$ on $(0,\infty)$ with periodic boundary conditions such that
\begin{equation}
\label{reg-weak}
\bfu \in \mathcal{C}([0,T]; \dot{L}_\sigma^2(\Omega)) \cap L^\frac{11}{5} (0 , T;  W_\sigma^{1,\frac{11}{5}}(\Omega)), \quad \partial_t \bfu \in L^{\frac{11}{6}}(0,T;(W_\sigma^{1,\frac{11}{5}}(\Omega))'), \quad 
\forall \, T \geq 0,
\end{equation}
and 
\begin{equation}
\label{weak-prob}
\lb \partial_t \bfu, \bfw \rb + ( (\bfu \cdot \nabla) \bfu, \bfw)   + ( \mathbf{T}(\D \bfu), \nabla \bfw)= (\bff, \bfw), \quad \forall \, \bfw \in W_\sigma^{1,\frac{11}{5}}(\Omega),
\end{equation}
for almost all $t\in [0,T]$. 
 Moreover,  the energy equality holds
\begin{equation}
\label{EnEq2}
\frac12 \| \bfu(t)\|^2+ \int_{0}^t \left( 2\nu_0 \|\D \bfu (\tau)\|^2  + 2\nu_1 \|\D \bfu (\tau)\|_{L^\frac{11}{5}}^\frac{11}{5} \right) \, d\tau 
= \frac12 \| \bfu_0\|^2 + \int_0^t (\bff(\tau),\bfu(\tau)) \, d \tau, \quad \forall \, t \geq 0.
\end{equation}
In particular, if $\bff\in \dot{L}^2(\Omega)$, there exists a time $t_0 > 0$
such that for all $t \geq t_0$ we have
\begin{equation}\label{L2Bound-2}
\|\bfu(t)\|^2 \leq  2 \frac{\nu_0^2 G^2}{\lambda_1^\frac12}
\end{equation}
and 
\begin{equation}\label{W1pBound-2}
\int_t^{t+T}\, \left( \nu_0\|\D \bfu(\tau)\|^2+ \nu_1\| \D \bfu(\tau)\|_{L^\frac{11}{5}}^\frac{11}{5}\right) \, d\tau \leq 
2\left(1+ \nu_0 \lambda_1 T \right) 
\frac{\nu_0^2 G^2}{\lambda_1^\frac12},
\end{equation}
where $G$ is defined as in \eqref{Grashof}.
In addition, there exists $\overline{t} \in [t_0,t_0+1]$ such that 
\begin{equation}
\label{reg-strong}
\bfu \in L^\infty(\overline{t},T;\dot{H}^1_\sigma(\Omega))\cap L^2(\overline{t},T;\dot{H}^2_\sigma(\Omega))\cap L^\frac{11}{5}(\overline{t},T;W^{1,\frac{33}{5}}(\Omega)), \quad
\forall \, T\geq \overline{t},
\end{equation}
and
\begin{equation}
\int_t^{t+r} \|\nabla \bfu(\tau) \|_{L^\frac{33}{5}}^\frac{11}{5} \, d \tau\leq \frac{1}{K_1} \left( R_3+K_2 R_2 R_3 +K_3 R_2 + \nu_0^2 \lambda_1^\frac12 G^2 \right), \quad \forall \, t \geq t_1,
\end{equation}
where $r=(\nu_0 \lambda_1)^{-1}$, $t_1=\overline{t}+r$. The constants $K_1$, $K_2$, $K_3$ are defined in \eqref{K13}, and $R_1, R_2, R_3$ are given in \eqref{R12}-\eqref{R3}.
\end{thm}


\begin{proof}
The first part of Theorem \ref{115} is proved in \cite{MNRR}*{Section 5} (see also \cite{MR2005}{Theorem 3.1}). Let us now consider a generic\footnote{Indeed, in the case $p \in [ \frac{11}{5}, \frac{5}{2})$, the weak solutions are not known to be unique (cf. \cite{MR2005}).} weak solution $\bfu$ to \eqref{Lady}$_{1-2}$ on $(0,\infty)$ satisfying
\eqref{reg-weak}, \eqref{weak-prob}, \eqref{EnEq2},  \eqref{L2Bound-2} and \eqref{W1pBound-2}. It follows from \eqref{W1pBound-2} that there exists $\overline{t}\in [t_0,t_0+1]$ such that 
$$
\| \D \bfu(\overline{t})\| \leq \left( 2\left(1+ \nu_0 \lambda_1 \right) 
\frac{\nu_0 G^2}{\lambda_1^\frac12} \right)^\frac12.
$$
Since $\bfu(\overline{t})\in \dot{H}^1_\sigma$, we infer from \cite{MNRR}*{Theorem 3.4, Theorem 4.5 and Remark 4.6} (see also \cite{MR2005}*{Theorem 4.1}) that there exists a unique strong solution $\widetilde{\bfu}$ on $[\overline{t},\infty)$ originating from $\bfu$
such that
$$
\widetilde{\bfu} \in L^\infty(\overline{t},T;\dot{H}^1_\sigma(\Omega))\cap L^2(\overline{t},T;\dot{H}^2_\sigma(\Omega))\cap L^\frac{11}{5}(\overline{t},T;W^{1,\frac{33}{5}}(\Omega)), \quad
\forall \, T\geq \overline{t},
$$
In addition, in light of the weak-strong uniqueness principle proved in \cite{MR2005}*{Theorem 5.2}, we infer that $\widetilde{\bfu}(t)= \bfu(t)$ for any $t\in [\overline{t},\infty)$. This, in turn, gives \eqref{reg-strong}. 

We now perform some formal Sobolev estimates whose rigorous justification can be performed through the Galerkin scheme.
By definition of the Stokes operator in the periodic setting, multiplying \eqref{Lady-DA}$_{1}$ by $-\Delta \bfu$ and integrating over $\Omega$, we obtain
\begin{equation}
\begin{split}
\frac12 \ddt \| \nabla \bfu\|^2 &+ \nu_0\|\Delta \bfu\|^2 +
2\nu_1 \int_{\Omega} \nabla \cdot (|\D \bfu|_F^{\frac{1}{5}} \D \bfu) \cdot \Delta \bfu \, d \bfx\\
&= - \int_{\Omega} \bff \cdot \Delta \bfu \, d \bfx+\int_{\Omega} (\bfu\cdot \nabla) \bfu \cdot \Delta \bfu \, d \bfx.
\end{split}
\end{equation}
Here we have used that $\nabla \cdot \left( (\nabla \bfu)^T \right)= \nabla (\nabla \cdot \bfu)=0$ by \eqref{Lady-DA}$_2$. A direct calculation shows that
\begin{equation}
\label{deri}
\partial_k (|\D \bfu|_F^n)= n |\D \bfu|_F^{n-2} \D \bfu : \D (\partial_k \bfu), \quad \forall \, n>0. 
\end{equation}
Using integration by parts and \eqref{deri} with $n=p-2$, we have for $p\geq 2$ 
\begin{equation}
\begin{split}
&\int_{\Omega} \nabla \cdot \left(|\D \bfu|_F^{p-2} \D \bfu \right) \cdot \Delta \bfu \, d \bfx
=\int_{\Omega} \partial_j \left(|\D \bfu|_F^{p-2} (\D \bfu)_{ij}\right) \partial_{kk} \bfu_i \, d \bfx\\
&\quad=- \int_{\Omega} |\D \bfu|_F^{p-2} (\D \bfu)_{ij} \partial_{kk} \partial_j \bfu_i \, d \bfx\\
&\quad= \int_{\Omega} \partial_k \left(|\D \bfu|_F^{p-2} (\D \bfu)_{ij}\right) \partial_k (\D \bfu)_{ij} \, d \bfx\\
&\quad= \int_{\Omega} \partial_k \left( |\D \bfu|_F^{p-2} \right) 
(\D \bfu)_{ij} \partial_k (\D \bfu)_{ij} \, d \bfx +
\int_{\Omega} |\D \bfu|_F^{p-2}
\partial_k (\D \bfu)_{ij} \partial_k (\D \bfu)_{ij} \, d \bfx\\
&\quad= \int_{\Omega} (p-2) |\D \bfu|_F^{p-4} (\D \bfu)_{lm} (\D \partial_k \bfu)_{lm} \, (\D \bfu)_{ij} (\D \partial_k \bfu)_{ij} \, d \bfx +
\int_{\Omega} |\D \bfu|_F^{p-2} 
|\nabla (\D \bfu) |^2 \, d \bfx\\
&\quad= \int_{\Omega} (p-2) |\D \bfu|_F^{p-4} |\D \bfu: \D (\nabla\bfu)|^2 \, d \bfx + \int_{\Omega} |\D \bfu|_F^{p-2} 
|\nabla (\D \bfu) |^2 \, d \bfx.
\end{split}
\end{equation}
Exploiting again \eqref{deri} with $n=\frac{p}{2}$, we observe that
$$
\int_{\Omega} |\nabla |\D \bfu|_F^{\frac{p}{2}}|^2 \, d \bfx = \left( \frac{p}{2}\right)^2 \int_{\Omega} |\D \bfu|_F^{p-4} |\D \bfu : \D (\nabla\bfu))|^2 \, d\bfx.
$$
As a consequence, it follows for $p=\frac{11}{5}$ that
\begin{align*}
\int_{\Omega} \nabla \cdot \left(|\D \bfu|_F^{\frac15} \D \bfu \right) \cdot \Delta \bfu \, d \bfx &\geq \frac15. \left( \frac{10}{11}\right)^2 \int_{\Omega} |\nabla |\D \bfu|_F^{\frac{11}{10}}|^2 \, d \bfx \\
&= \frac{20}{121} \left\| |\D \bfu|_F^\frac{11}{10}\right\|_{H^1}^2 - \frac{20}{121}
\left\| |\D \bfu|_F^{\frac{11}{10}}\right\|^2\\
&\geq \frac18
\left\| |\D \bfu|_F^\frac{11}{10}\right\|_{H^1}^2 - \frac16
\left\| \D \bfu \right\|_{L^\frac{11}{5}}^\frac{11}{5}.
\end{align*}
Using the embedding $H^1(\Omega)\hookrightarrow L^6(\Omega)$ and the Korn inequality,  we infer that
\begin{align*}
\int_{\Omega} \nabla \cdot \left(|\D \bfu|_F^{\frac15} \D \bfu \right) \cdot \Delta \bfu \, d \bfx
&\geq \frac18 \frac{1}{c_S^2} \left\| |\D \bfu|_F^\frac{11}{10}\right\|_{L^6}^2 - \frac16
\left\| \D \bfu \right\|_{L^p}^p\\
&\geq \frac18 \frac{C^\frac{11}{5}}{c_S^2} \left\| \D \bfu\right\|_{L^{\frac{33}{5}}}^\frac{11}{5} - \frac{1}{6}\left\| \D \bfu \right\|_{L^\frac{11}{5}}^\frac{11}{5}\\
&\geq \frac18 \frac{C^\frac{11}{5}}{c_S^2 \, c_K^\frac{11}{5}} \left\| \nabla \bfu\right\|_{L^{\frac{33}{5}}}^\frac{11}{5} - \frac16 \left\| \D \bfu \right\|_{L^\frac{11}{5}}^\frac{11}{5}.
\end{align*} 
In order to handle the convective term, we observe that 
\begin{equation}
\begin{split} 
\int_{\Omega} (\bfu\cdot \nabla) \bfu \cdot \Delta \bfu \, d \bfx
&= \int_{\Omega} \bfu_j \partial_j \bfu_i \partial_{kk} \bfu_i \, d \bfx\\
&=-\int_{\Omega} \partial_k \bfu_j \partial_j \bfu_i \partial_k \bfu_i \, d \bfx - \int_{\Omega} \bfu_j \partial_j \partial_k \bfu_i \partial_k \bfu_i \, d \bfx\\
&= -\int_{\Omega} \partial_k \bfu_j \partial_j \bfu_i \partial_k \bfu_i \, d \bfx - \underbrace{\int_{\Omega} \bfu_j \partial_j \left( \frac12 \partial_k \bfu_i \partial_k \bfu_i  \right) \, d \bfx}_{=0}\leq \| \nabla \bfu \|_{L^3}^3.
\end{split}
\end{equation}
Thus, collecting the above terms together, we find the differential inequality
\begin{equation}
\label{S-di-1}
\begin{split}
&\frac12 \ddt \| \nabla \bfu\|^2 + \nu_0\|\Delta \bfu\|^2 + \frac{\nu_1 \widetilde{C}}{4}
 \left\| \nabla \bfu\right\|_{L^{\frac{33}{5}}}^\frac{11}{5} 
\leq \| \nabla \bfu \|_{L^3}^3 + \frac{\nu_1}{3} \left\| \D \bfu \right\|_{L^\frac{11}{5}}^\frac{11}{5} - \int_{\Omega} \bff \cdot \Delta \bfu \, d \bfx.
 \end{split}
\end{equation}
Here, we have set $\widetilde{C}=\frac{C^\frac{11}{5}}{c_S^2 \, c_K^\frac{11}{5}}$, which depends only on $\Omega$ and the value $p=\frac{11}{5}$.
We now proceed with the estimate of the terms on the right-hand side of \eqref{S-di-1}. We exploit the splitting method devised in \cite{MNRR} for the $L^3$-norm of $\nabla \bfu$ which follows from the Lebesgue interpolation. We recall that for $p \in [2, 3]$
$$
\| \nabla \bfu\|_{L^3}\leq \| \nabla \bfu\|_{L^p}^\frac{p-1}{2} \| \nabla \bfu\|_{L^{3p}}^\frac{3-p}{2},
\quad
\| \nabla \bfu\|_{L^3}\leq \| \nabla \bfu\|_{L^2}^\frac{2p-2}{3p-2} \| \nabla \bfu\|_{L^{3p}}^\frac{p}{3p-2}.
$$
For $\alpha \in (0,1)$, which will be chosen later, exploiting the above interpolation inequalities, we obtain
\begin{equation}
\begin{split}
\| \nabla \bfu\|_{L^3}^3
&\leq \| \nabla \bfu \|_{L^3}^{3\alpha} \| \nabla \bfu\|_{L^3}^{3(1-\alpha)} \\
& \leq \| \nabla \bfu \|_{L^p}^{3\alpha \frac{p-1}{2}} 
\| \nabla \bfu\|_{L^{3p}}^{3\alpha \frac{3-p}{2}}
\| \nabla \bfu\|_{L^2}^{3 (1-\alpha) \frac{2p-2}{3p-2}}
\| \nabla \bfu\|_{L^{3p}}^{3(1-\alpha)\frac{p}{3p-2}}\\
& \leq \| \nabla \bfu \|_{L^p}^{3\alpha \frac{p-1}{2}} 
\| \nabla \bfu\|_{L^2}^{3 (1-\alpha) \frac{2p-2}{3p-2}}
\| \nabla \bfu\|_{L^{3p}}^{3\alpha \frac{3-p}{2}+3(1-\alpha)\frac{p}{3p-2}}.
\end{split}
\end{equation}
In particular, for $p=\frac{11}{5}$, we have
$$
\| \nabla \bfu\|_{L^3}^3
\leq \| \nabla \bfu \|_{L^\frac{11}{5}}^{\frac{9}{5} \alpha} 
\| \nabla \bfu\|_{L^2}^{\frac{36}{23} (1-\alpha)} 
\| \nabla \bfu\|_{L^{\frac{33}{5}}}^{\frac{33}{23}-\alpha \frac{27}{115}}.
$$
Setting
$$
\alpha= \frac{22}{45}, \quad s=\frac{5}{3},\quad s'= \frac{5}{2},
$$
and using the Young inequality, it follows that for any $\varepsilon>0$ 
\begin{equation}
\begin{split}
\| \nabla \bfu\|_{L^3}^3 
&\leq 
\| \nabla \bfu \|_{L^\frac{11}{5}}^{\frac{22}{25}} 
\| \nabla \bfu\|_{L^2}^{\frac{4}{5}} 
\| \nabla \bfu\|_{L^{\frac{33}{5}}}^{\frac{33}{25}}\\
& \leq \frac{3\varepsilon}{5} \| \nabla \bfu\|_{L^{\frac{33}{5}}}^\frac{11}{5}
+ \frac{2}{5 \varepsilon^{\frac{3}{2}}} \| \nabla \bfu\|_{L^\frac{11}{5}}^\frac{11}{5} \| \nabla \bfu\|_{L^2}^2.
\end{split}
\end{equation}
Choosing $\varepsilon= \frac{5}{24} \nu_1 \widetilde{C}$, we are led to
\begin{equation}
\label{Dv-L3}
\| \nabla \bfu\|_{L^3}^3 
\leq 
\frac{\nu_1 \widetilde{C}}{8}
 \left\| \nabla \bfu\right\|_{L^{\frac{33}{5}}}^\frac{11}{5}
+ \frac25 \left( \frac{24}{5\nu_1 \widetilde{C}} \right)^\frac32 \| \nabla \bfu\|_{L^\frac{11}{5}}^\frac{11}{5} \| \nabla \bfu\|_{L^2}^2.
\end{equation}
Also, we have
\begin{equation}
\label{fDv}
- \int_{\Omega} \bff \cdot \Delta \bfu \, d \bfx 
\leq \frac{\nu_0}{2} \| \Delta \bfu\|^2 
+ \frac{1}{2\nu_0} \| \bff\|^2.
\end{equation}
Combining \eqref{S-di-1} with \eqref{Dv-L3} and \eqref{fDv}, we end up with 
\begin{equation}
\label{S-di-2}
\begin{split}
&\frac12 \ddt \| \nabla \bfu\|^2 + \frac{\nu_0}{2}\|\Delta \bfu\|^2 +  \frac{\nu_1 \widetilde{C}}{8}
 \left\| \nabla \bfu\right\|_{L^{\frac{33}{5}}}^\frac{11}{5} \\
&\leq \frac25 \left( \frac{24}{5\nu_1 \widetilde{C}} \right)^\frac32 \| \nabla \bfu\|_{L^\frac{11}{5}}^\frac{11}{5} \| \nabla \bfu\|_{L^2}^2 + \frac{\nu_1}{3} \left\| \D \bfu \right\|_{L^\frac{11}{5}}^\frac{11}{5}+ \frac{1}{2 \nu_0} \| \bff\|^2,
 \end{split}
\end{equation}
for almost any $t \in (\overline{t},\infty)$.
We rewrite the above inequality as
\begin{equation}
\label{S-di-3a}
\begin{split}
&\ddt \| \nabla \bfu\|^2 + \nu_0\|\Delta \bfu\|^2 +  K_1
 \left\| \nabla \bfu\right\|_{L^{\frac{33}{5}}}^\frac{11}{5}\leq K_2 \| \nabla \bfu\|_{L^\frac{11}{5}}^\frac{11}{5} \| \nabla \bfu\|_{L^2}^2 + K_3 \left\| \nabla \bfu \right\|_{L^\frac{11}{5}}^\frac{11}{5}+ \frac{1}{\nu_0} \| \bff\|^2,
 \end{split}
\end{equation}
having set
\begin{equation}
\label{K13}
K_1= \frac{\nu_1 \widetilde{C}}{4},
\quad 
K_2=\frac25 \left( \frac{24}{5\nu_1 \widetilde{C}} \right)^\frac32, 
\quad 
K_3= \frac{2 \nu_1 C}{3}.
\end{equation}
In particular, we have
\begin{equation}
\label{S-di-3b}
\begin{split}
&\ddt \| \nabla \bfu\|^2 \leq K_2 \| \nabla \bfu\|_{L^\frac{11}{5}}^\frac{11}{5} \| \nabla \bfu\|_{L^2}^2 + K_3 \left\| \nabla \bfu \right\|_{L^\frac{11}{5}}^\frac{11}{5}+ \frac{1}{\nu_0} \| \bff\|^2.
 \end{split}
\end{equation}
In light of \eqref{W1pBound-2}, for any $t\geq t_0$ and $r=(\nu_0 \lambda_1)^{-1}$ we infer that 
\begin{equation}
\label{R12}
\int_{t}^{t+r} 
\| \nabla \bfu(\tau)\|^2 \, d \tau \leq
8 \frac{\nu_0 G^2}{\lambda_1^\frac12}=: R_1,
\quad 
\int_{t}^{t+r} 
\| \nabla \bfu(\tau)\|_{L^\frac{11}{5}}^\frac{11}{5} \, d \tau \leq
4 c_K^\frac{11}{5} \frac{\nu_0^2 G^2}{\nu_1 \lambda_1^\frac12}=: R_2.
\end{equation}
By exploiting Lemma \ref{unif-gronw}, we find
\begin{equation}
\label{R3}
\| \nabla \bfu(t)\|^2 \leq
\left( \nu_0 \lambda_1 R_1 + K_3 R_2+ \nu_0^2 \lambda_1^\frac12 G^2 \right) \mathrm{e}^{K_2 R_2}=:R_3, \quad \forall \, t\geq \overline{t}+r=t_1.
\end{equation}
As an immediate consequence, integrating \eqref{S-di-2} from $t$ to $t+r$, where $t\geq t_1$, we obtain
\begin{equation}
\label{335}
\int_t^{t+r} \|\nabla \bfu(\tau) \|_{L^\frac{33}{5}}^\frac{11}{5} \, d \tau\leq \frac{1}{K_1} \left( R_3+K_2 R_2 R_3 +K_3 R_2 + \nu_0^2 \lambda_1^\frac12 G^2  \right).
\end{equation}
\end{proof}

Next, we state the following result concerning the existence of solutions to the data assimilation algorithm given by \eqref{Lady-DA} in the case $p=\frac{11}{5}$. This is a consequence of the results obtained in \cites{MR2005,MNRR}.

\begin{thm}[\textbf{Existence of weak and strong solutions for data assimilation problem}]
\label{DA115}
Assume that $p=\frac{11}{5}$ and $\bff \in \dot{L}^2(\Omega)$. Let $\bfu$ be a weak solution of \eqref{Lady} with periodic boundary conditions given by Theorem \ref{115}. Then, we have the following:
\begin{itemize}
    \item[1.] If $\bfv_0\in \dot{L}_\sigma^2(\Omega)$, there exists a weak solution $\bfv$ to \eqref{Lady-DA}  satisfying
    \begin{equation}
     \label{reg-weak-v115b}
\bfv \in \mathcal{C}([0,T]; \dot{L}_\sigma^2(\Omega)) \cap L^\frac{11}{5} (0 , T;  W_\sigma^{1,\frac{11}{5}}(\Omega)), \quad \partial_t \bfv \in L^{\frac{11}{6}}(0,T;(W_\sigma^{1,\frac{11}{5}}(\Omega))'), \quad 
\forall \, T \geq 0,
\end{equation}
and 
\begin{equation}
\label{weak-prob-v115}
\lb \partial_t \bfv, \bfw \rb + ( (\bfv \cdot \nabla) \bfv, \bfw)   + ( \mathbf{T}(\D \bfv), \nabla \bfw)= (\bff, \bfw)-\mu (I_h(\bfv-\bfu),\bfw), \quad \forall \, \bfw \in \dot{W}_\sigma^{1,\frac{11}{5}}(\Omega),
\end{equation}
for almost all $t\in [0,T]$.
\medskip

\item[2.] If $\bfv_0\in \dot{H}_\sigma^1(\Omega)$, there exists a unique strong solution $\bfv$ to \eqref{Lady-DA} such that
\begin{equation}
     \label{reg-strong-v115}
\bfv \in \mathcal{C}([0,T]; \dot{H}^1_\sigma(\Omega)) \cap L^2 (0 , T;  H_\sigma^2(\Omega))\cap L^\frac{11}{5}(0,T;W^{1, \frac{33}{5}}(\Omega)), \quad
\forall \, T \geq 0,
\end{equation}
which solves \eqref{Lady-DA} in weak sense as in \eqref{weak-prob-v115}.
\medskip

\item[3.] If $\bfv_0 \in \dot{W}^{1,\frac{11}{5}}(\Omega)$, there exists a unique strong solution $\bfv$ to \eqref{Lady-DA} which satisfies, in addition to \eqref{reg-strong-v115},
\begin{equation}
     \label{reg-weak-v115a}
\bfv \in \mathcal{C}([0,T]; \dot{W}^{1,\frac{11}{5}}_\sigma(\Omega)), \quad \partial_t \bfv \in L^2(0,T;\dot{L}_\sigma^2(\Omega), \quad 
\forall \, T \geq 0.
\end{equation}
In particular, in this case 
$\bfv$ solves \eqref{Lady-DA} in weak sense with $\bfw \in \dot{H}_\sigma^1(\Omega)$.
\end{itemize}
\end{thm}

Lastly, we prove the convergence result for $p=\frac{11}{5}$ in the periodic boundary setting.

\begin{thm}\label{thm4}
For $p=\frac{11}{5}$, let $\bfu$ be a weak solution of \eqref{Lady} with periodic boundary conditions given by Theorem \ref{115} and let $\bfv$ be the solution to the data assimilation algorithm given by Theorem \eqref{DA115}. Assume that 
\begin{equation}
\label{mu115}
\mu \geq  \frac{2 \overline{C} \nu_0^\frac{5}{17} \lambda_1^\frac{10}{17}}{K_1^\frac{10}{17}} \left( R_3+K_2 R_2 R_3 +K_3 R_2 + \nu_0^2 \lambda_1^\frac12 G^2 \right)^{\frac{10}{17}}
\end{equation}
where $\overline{C}$ is a constant depending on $\Omega$ and $p$ and $K_1, K_2, K_3, R_2, R_3$ are defined in Theorem \ref{115}, and $h$ small enough such that 
 $$\mu \,  c_0\,  h^2 \leq \nu_0,$$
 where $c_0$ is a dimensionless constant given \eqref{I_h}. Then, we have 
$$
\| \bfu(t) - \bfv(t)\|_{L^2(\Omega)} \rightarrow 0,
$$
at exponential rate, as $t \rightarrow \infty$. 
\end{thm}
\begin{proof}
Proceeding as in the proof of Theorem \ref{thm2},  we have
 \begin{equation} \label{thm2Eq5-2}
\frac{1}{2}  \ddt \|\bfe\|^2 +  \frac{\nu_0}{2} \|\nabla \bfe\|^2  \leq  | \left( \left(\bfe \cdot \nabla\right) \bfu , \bfe\right)  |-  \frac{\mu}{2} \|\bfe\|^2.
\end{equation}
Arguing differently than \eqref{thm2Eq6}, we find
\begin{equation}\label{thm2Eq6-2}
\begin{split}
 | \left( \left(\bfe \cdot \nabla\right) \bfu , \bfe\right) & | \leq   \| \bfe^2\|_{L^{\frac{33}{28}}}  \, \|\nabla \bfu\|_{L^\frac{33}{5}}
 = \| \bfe\|^2_{L^{\frac{33}{14}}}\, \|\nabla \bfu\|_{L^\frac{33}{5}}  
 \leq  \|\bfe\|^{\frac{17}{11}}\, \|\bfe\|_{L^6}^{\frac{5}{11}} \,  \|\nabla \bfu\|_{L^\frac{33}{5}} \\
 & \leq  c_S^\frac{5}{11} \|\bfe\|^{\frac{17}{11}}\, \| \nabla \bfe\|^{\frac{5}{11}} \, \|\nabla \bfu\|_{L^\frac{33}{5}}  
 \leq \frac{\nu_0}{2}  \| \nabla \bfe\|^2 
 + c_S^\frac{10}{17} \left( \frac{2}{\nu_0}\right)^\frac{5}{17}  \,  \|\nabla \bfu\|_{L^\frac{33}{5}}^\frac{22}{17} \, \|\bfe\|^2.
\end{split}
\end{equation}
Inserting \eqref{thm2Eq6-2} in \eqref{thm2Eq5-2}, we arrive at
\begin{equation}\label{thm2Eq7-2}
\ddt \|\bfe\|^2 + \left( \mu -  \frac{\overline{C}}{\nu_0^\frac{5}{17}}  \,  \|\nabla \bfu\|_{L^\frac{33}{5}}^\frac{22}{17}\right)\,  \|\bfe\|^2 \leq 0,
\end{equation}
for some constant $\overline{C}$ depending only on $\Omega$ and the value $p=\frac{11}{5}$.
Aiming to use Lemma \ref{Gronwall2}, let us set 
$$
\alpha(t) =  \left( \mu -  \frac{\overline{C}}{\nu_0^\frac{5}{17}}  \,  \|\nabla \bfu\|_{L^\frac{33}{5}}^\frac{22}{17}\right).
$$
By H\"older's inequality and \eqref{335}, we obtain 
\begin{equation*}
\begin{split}
\int_t^{t+r}\, \alpha(s)\, ds & = \mu r  - \frac{\overline{C}}{\nu_0^\frac{5}{17}}  \, \int_t^{t+r}\,  \|\nabla \bfu (s)\|_{L^\frac{33}{5}}^{\frac{22}{17}} ds\\
& \geq \frac{\mu}{\nu_0 \lambda_1}  - \frac{\overline{C}}{\nu_0^\frac{5}{17}}
\,  \left(\int_t^{t+r} \|\nabla \bfu (s)\|_{L^\frac{33}{5}}^{\frac{11}{5}} \, ds\right)^{\frac{10}{17}} \left(\int_t^{t+r} 1 \, ds\right)^{\frac{7}{17}}\\
& \geq \frac{\mu}{\nu_0 \lambda_1}  - \frac{\overline{C}}{\nu_0^\frac{12}{17} \lambda_1^\frac{7}{17}}
\, \frac{1}{K_1^\frac{10}{17}} \left( R_3+K_2 R_2 R_3 +K_3 R_2 + \nu_0^2 \lambda_1^\frac12 G^2 \right)^{\frac{10}{17}}.
\end{split}
\end{equation*}
Notice that the second term on the right-hand side of the above inequality is independent of $\mu$. In particular, in light of the assumption \eqref{mu115} we immediately deduce that
$$ 
\int_t^{t+r}\, \alpha(s)\, ds >0, \quad \forall \, t\geq t_1.
$$
Therefore,  we conclude from Lemma \ref{Gronwall2} that $\|\bfe\| = \| \bfu - \bfv\| \rightarrow 0$ exponentially fast as $t \rightarrow \infty$.
\end{proof}

\begin{remark}[2D case]
The condition \eqref{mu115} for the nudging parameter $\mu$ can be enhanced in 2D. Indeed, recalling that $\int_{\Omega} (\bfu\cdot \nabla) \bfu \cdot \Delta \bfu \, d \bfx= 0$, \eqref{S-di-3a} is replaced by 
\begin{equation}
\begin{split}
&\ddt \| \nabla \bfu\|^2 + \nu_0\|\Delta \bfu\|^2 +  2 K_1
 \left\| \nabla \bfu\right\|_{L^{\frac{33}{5}}}^\frac{11}{5}\leq 
  K_3 \left\| \nabla \bfu \right\|_{L^\frac{11}{5}}^\frac{11}{5}+ \frac{1}{\nu_0} \| \bff\|^2.
 \end{split}
\end{equation}
Then, arguing as in the proof of Theorem \ref{115}, it follows that
\begin{equation}
\label{R3-star}
\| \nabla \bfu(t)\|^2 \leq
\left( \nu_0 \lambda_1 R_1 + K_3 R_2+ \nu_0^2 \lambda_1^\frac12 G^2 \right) =:R_3^\star, \quad \forall \, t\geq \overline{t}+r=t_1,
\end{equation}
and 
\begin{equation}
\int_t^{t+r} \|\nabla \bfu(\tau) \|_{L^\frac{33}{5}}^\frac{11}{5} \, d \tau\leq \frac{1}{2 K_1} \left( R_3^\star+K_3 R_2 + \nu_0^2 \lambda_1^\frac12 G^2  \right), \quad \forall \, t\geq t_1.
\end{equation}
As a direct consequence, \eqref{mu115} becomes
\begin{equation}
\label{mu115-2D}
\mu \geq  \frac{2 \overline{C} \nu_0^\frac{5}{17} \lambda_1^\frac{10}{17}}{(2\, K_1)^\frac{10}{17}} \left( R_3^\star+K_3 R_2 + \nu_0^2 \lambda_1^\frac12 G^2 \right)^{\frac{10}{17}}.
\end{equation}
Furthermore, the analysis herein presented can be extended for any $p>2$ in \eqref{T-nonlinear}.
\end{remark}

\section{Computational Results}

We demonstrate the effectiveness of nudging for both two and three-dimensional Ladyzhenskaya models with fully periodic boundary conditions  in  $\Omega=[0,2\pi]^d$, $d=2,3$. {This is first done for the case $p=3$,  the Smagorinsky model, which is often used in Large Eddy Simulation (LES) of turbulent flow \cites{BIL06, J04}. We then vary $p$ in the three-dimensional case, and test nudging with only the horizontal components of velocity.} For both cases, the parameter $\nu_1$ is chosen from dimensional considerations to be 
\begin{align}\label{p-nu1rel}
\nu_1=\frac12(C_s\delta)^2\nu_0^{3-p}\ ,\quad C_s=0.1\ ,\quad \delta=\frac{2\pi}{N}\ ,
\end{align}
where $N$ is the number of Fourier modes used in each direction for the direct numerical simulation (DNS) of the reference solution.

The initial condition for the reference solution $\ubf(t_0)$ for each data assimilation experiment is chosen so that it faithfully reflects the long term dynamics of the model. This is done by integrating the model starting at $t=0$ with $\ubf(0)=0$ until some time $t=t_0$ when it appears the transient period has passed. { Figure \ref{fig:energy-time} shows the time evolution of the energy $\|\ubf\|_{L^2}^2$ on $[0,t_0]$. By the end of the run, this quantity seems to have reached its statistically stationary state. We assume then that $\ubf(t_0)$ is essentially on the global attractor. 
We start the nudging at time $t=t_0$ by solving the original ($\ubf$) and the nudging ($\vbf$) systems simultaneously with $\vbf(t_0)=0$.}  The computations are done using Dedalus,  an open-source spectral package (see \cite{dedalus2020}). The time stepper is a four-stage third order Runge-Kutta method.

\begin{figure}[!htbp]
         \subcaptionbox{2D Smagorinsky ($512\times 512$)}{\includegraphics{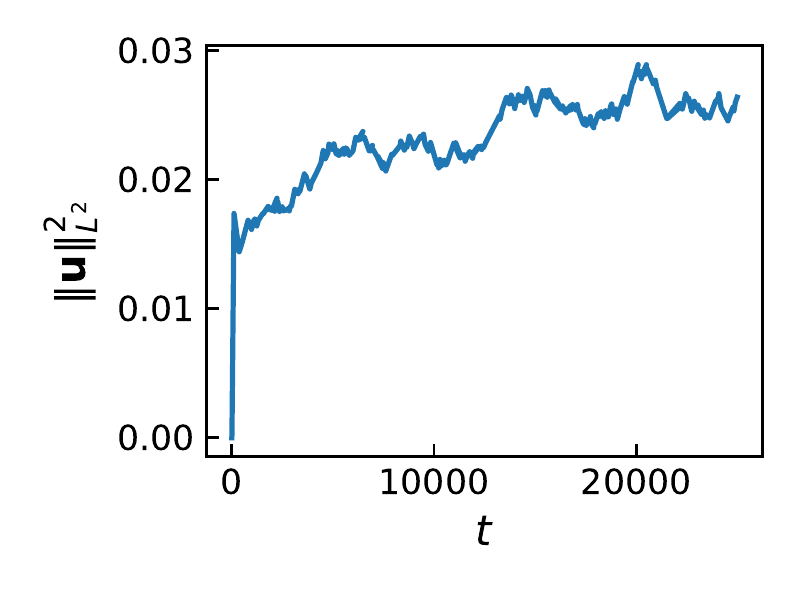}}
         \subcaptionbox{3D Smagorinsky ($256\times 256\times 256$)}{\includegraphics{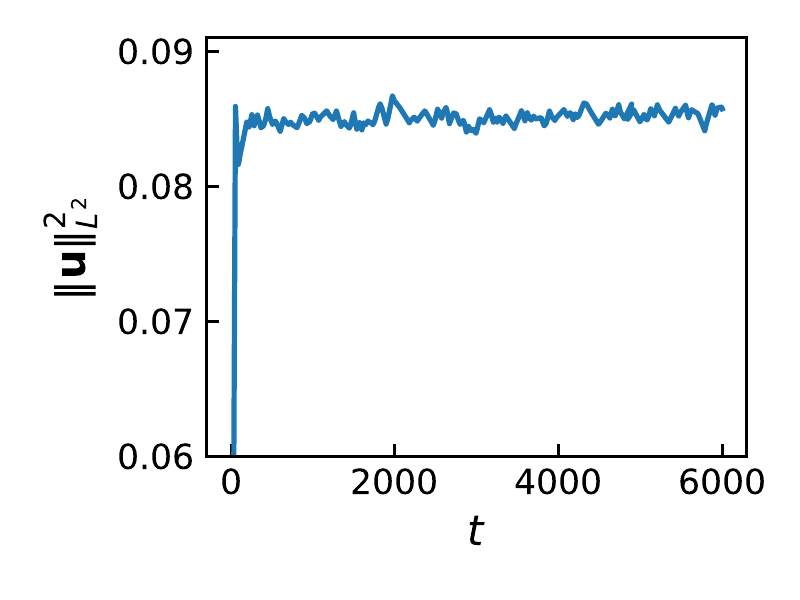}}
         \caption{Evolution of the energy of the reference solution over the transient period.
         }
                 \label{fig:energy-time}
\end{figure}

\subsection{Two-dimensional case}
{In two-dimensions, we take the viscosity to be $\nu_0=10^{-4}$, $\mu=1$, and use a normalized force $\bff_{\mathrm{2D}}$ from \cite{olson2008determining}, so that the Grashof number $G=2.5\times 10^5$.
We demonstrate both the nodal value and Fourier modes interpolant operators. In the nodal value case, we use every $4$th nodal value in each direction so that $h\approx 0.0491$. In the Fourier modes case, we use the projection on the low modes with wave vectors $\mathbf{k}=(k_1,k_2)$ such that $|k_j|\le 32$ and
$
h=\frac{\pi}{32}\approx 0.0982\;.
$
The value of $N$ is fixed at $512$.  While we have not analyzed the nodal interpolation operator in this paper, Figure \ref{2d-da} shows synchronization with the DNS of the reference solution to within machine precision in both the $L^2$ and $H^1$ norms. The same is true for Fourier mode interpolation, with a slower rate due to a larger value of $h$.} Field plots of the velocity components and pressures at several times near the start of nudging  corresponding to Figure \ref{2d-da-nodal} are shown in Figure \ref{fig:2d-da}.

\begin{figure}[!htbp]
         \subcaptionbox{Nodal interpolation\label{2d-da-nodal}}{\includegraphics{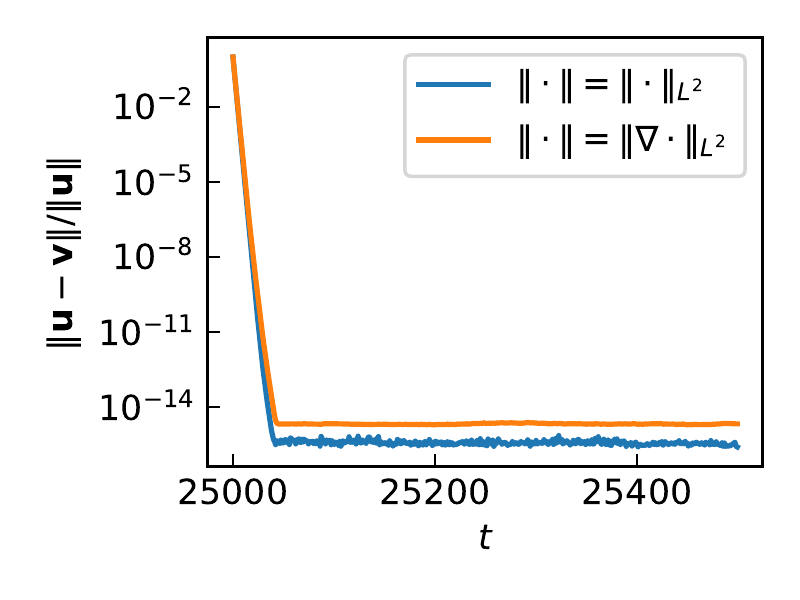}}
         \subcaptionbox{Fourier interpolation }{\includegraphics{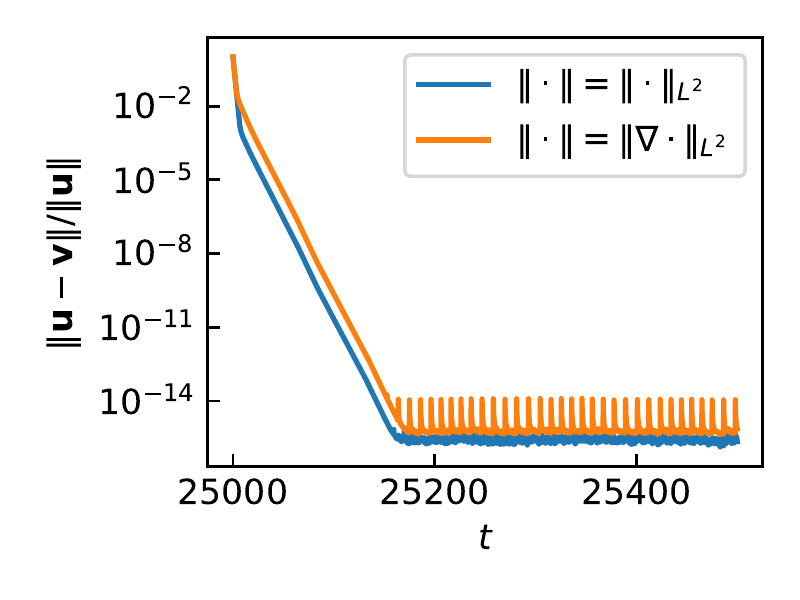}}
         \caption{Convergence of data assimilation for 2D Smagorinsky model for $\mu=1$.}\label{2d-da}
\end{figure}

\subsection{Three-dimensional case}

In the three-dimensional case, we take the function $g:=\nabla\times \bff_{\mathrm{2D}}$ and define a force $\bff_{\mathrm{3D}}=(f_1,f_2,f_3)$
 so that in each wave vector plane, $\bff_{\mathrm{3D}}$ is similar to $\bff_{\mathrm{2D}}$. Specifically, the 
 nonzero Fourier coefficients are:
 \begin{align*}
    \hat{f_1}(k_1,0,k_3)=\frac{ik_3\hat{g}(k_1,k_3)}{k_1^2+k_3^2}\, ,\quad  &\hat{f_1}(k_1,k_2,0)=\frac{ik_2\hat{g}(k_1,k_2)}{k_1^2+k_2^2}\, ,\\
    \hat{f_2}(k_1,k_2,0)=\frac{-ik_1\hat{g}(k_1,k_2)}{k_1^2+k_2^2}\, ,\quad &\hat{f_2}(0,k_2,k_3)=\frac{ik_3\hat{g}(k_2,k_3)}{k_2^2+k_3^2}\, ,\\
    \hat{f_3}(k_1,0,k_3)=\frac{-ik_1\hat{g}(k_1,k_3)}{k_1^2+k_3^2}\, ,\quad  &\hat{f_3}(0,k_2,k_3)=\frac{-ik_2\hat{g}(k_1,k_2)}{k_2^2+k_3^2}\, .
 \end{align*}
In 3D it is the viscosity $\nu_0$ that is adjusted so that the Grashof number remains as $G=2.5\times 10^5$.
We use the Fourier modes interpolation operator $I_h=P_{h(m)}$ for the 3D model, where $P_{h(m)}$ denotes the projection on the low modes with wave vectors $\mathbf{k}=(k_1,k_2,k_3)$ such that $|k_j|\le m$ and
$$
h(m)=\frac{\pi}{m}\;.
$$
{The value of $N$ is fixed at $256$.}

{Figure \ref{fig:3d-da-err} shows the exponential rate of synchronization using different values of nudging parameter $\mu$ and resolution $h$.  For fixed $\mu=10$, as we use fewer number of modes, the convergence is slower, but still exponential. {For $m=8$, slices of solutions at the mid-plane $z=\pi$ near the start of nudging are shown in Figure \ref{fig:3d-da}}. The convergence fails at $m=4$ (not shown).

At the fixed parameter of $m=32$, the convergence rate improves as $\mu$ is increased through $\mu=1$. (see Figure \ref{fig:3d-da-err-B}). At $\mu=1$ and $\mu=5$, the convergence rates are nearly identical, while at $\mu=0.01$, nudging fails to synchronize. This demonstrates a critical value of $\mu$.}

\begin{figure}[!htbp] 
 \subcaptionbox{\label{fig:3d-da-err-A}}
{\includegraphics{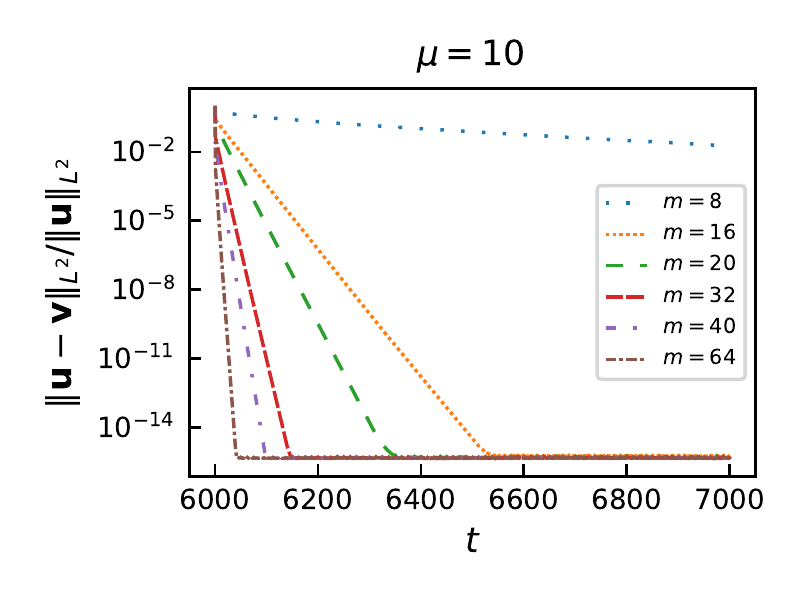}}
 \subcaptionbox{\label{fig:3d-da-err-B}}
{\includegraphics{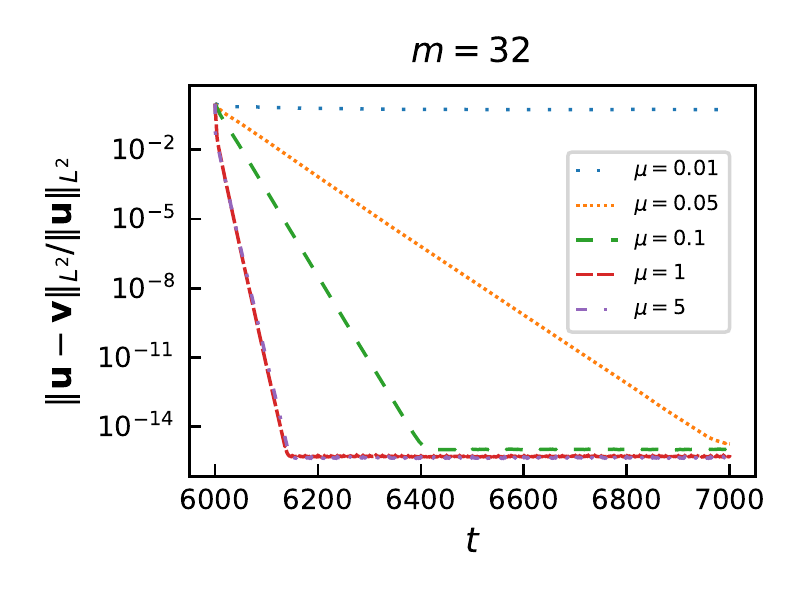}}
\caption{Convergence of data assimilation for the 3D Smagorinsky at different values of the nudging parameter $\mu$ and $h=h(m)$; the left fixes $\mu=10$ and the right fixes $m=32$.}
\label{fig:3d-da-err}
\end{figure}

\begin{figure}[!htbp]
    \centering
    \includegraphics[width=.8\textwidth]{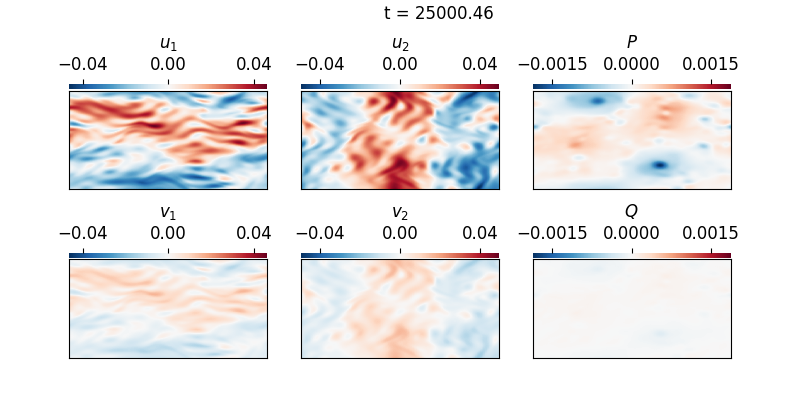}
    \includegraphics[width=.8\textwidth]{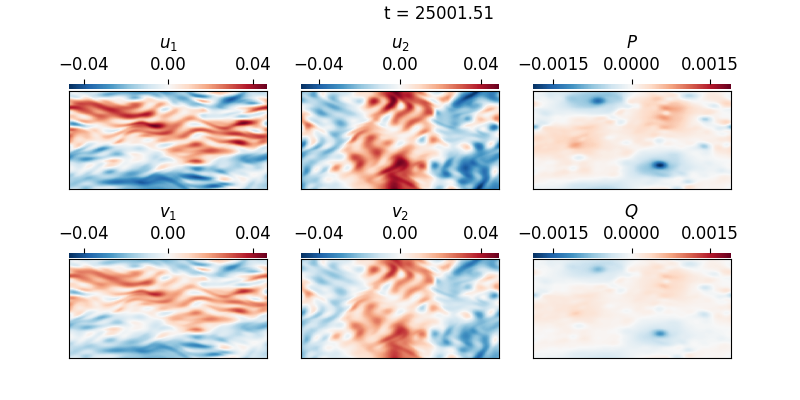}
    \includegraphics[width=.8\textwidth]{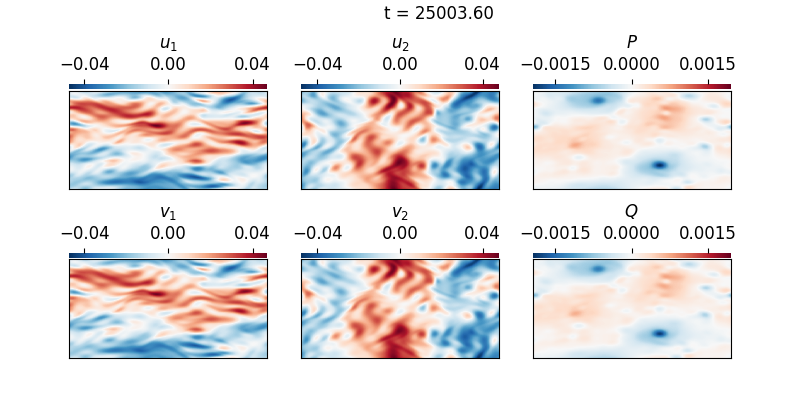}
    \caption{Synchronization of the 2D Smagorinsky model using nodal interpolation, $\mu=1$ and $h\approx 0.0491$; the reference solution $(\ubf,P)$ is denoted as $(u_1,u_2,P)$ and the nudging solution $(\vbf,Q)$ is $(v_1,v_2,Q)$.}
        \label{fig:2d-da}
\end{figure}

\begin{figure}[!htp]
         \centering
         \includegraphics[width=.95\textwidth]{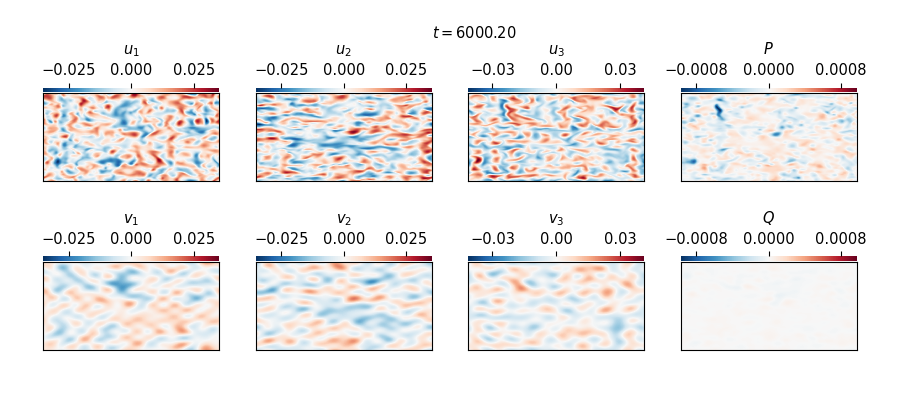}
         \includegraphics[width=.95\textwidth]{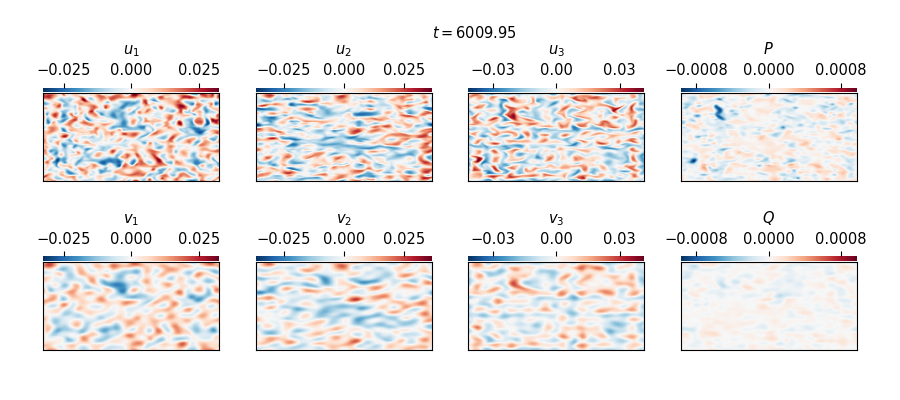}
         \includegraphics[width=.95\textwidth]{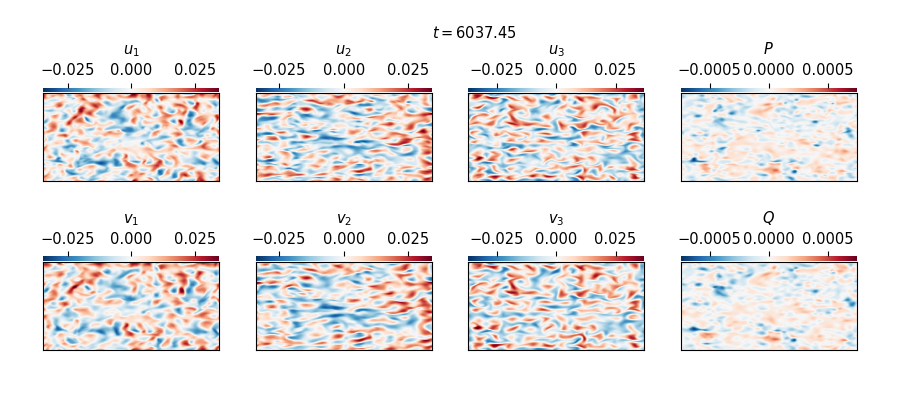}

\caption{Synchronization of the 3D Smagorinsky model using $\mu=10, h=h(8)$. These are the slices in the mid-plane $(0,2\pi)\times(0,2\pi)\times\{z=\pi\}$.}
         \label{fig:3d-da}
\end{figure}
We varied $p$ (along with $\nu_1$ according to \eqref{p-nu1rel}) in the Ladyzhenskaya model using both $\mu=10$ and $\mu=0.1$ (see Figure \ref{fig:3d-lady-p}).  At these values of $\mu$, we detect no discernible difference in the performance of the nudging algorithm for $p$ ranging from $2.2=11/5$ to 3.
\begin{figure}[!htp]
\centering
         \includegraphics{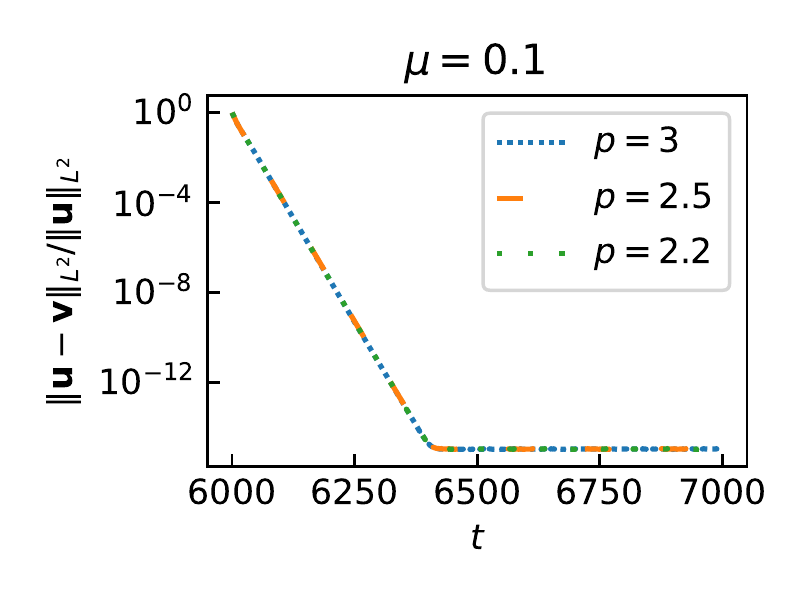}
         \includegraphics{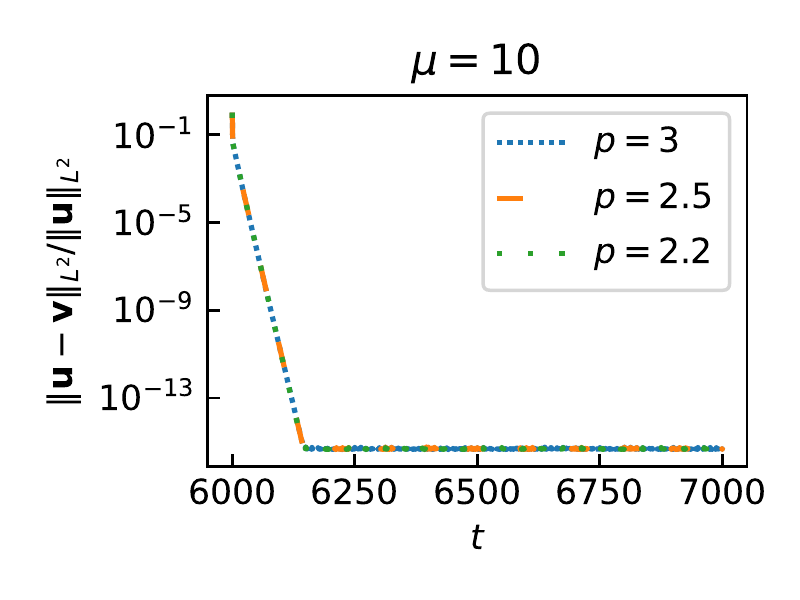}
\caption{Synchronization for the 3D Ladyzhenskaya model using $h=h(32)$ for different values of $p$.}
\label{fig:3d-lady-p}
\end{figure}

Finally, we consider an abridged nudging scheme in which only the horizontal components of velocity play the role of observed data.  This amounts to treating $\mu$ as the vector $(\mu_1, \mu_2, \mu_3)=(10,10,0)$ and nudging the $j^{\text{th}}$ component of velocity with the factor $\mu_j$.  Figure \ref{fig:3d-par-nud} shows rapid initial synchronization, which then slows, particularly for the third component of velocity, which is not nudged. While the error is far from machine precision even after nudging for 1000 time units, the field plots
shown in Figure \ref{fig:3d-da-abr} display similar features at rates that are slower for the third component of velocity and pressure.
\begin{figure}[ht]
\centering
         \includegraphics[width=.45\textwidth]{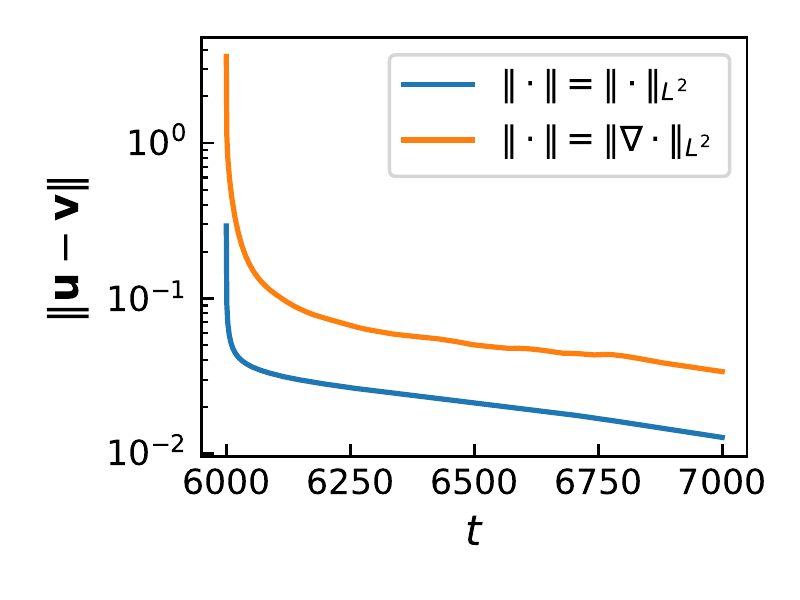}
         \includegraphics[width=.45\textwidth]{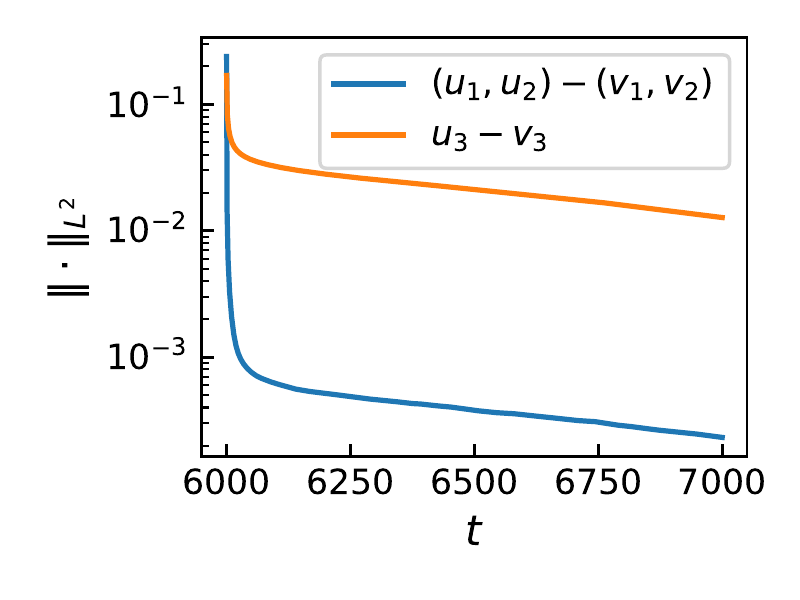}
\caption{Abridged nudging for the 3D Smagorinsky model with $(\mu_{1},\mu_{2},\mu_3)=(10,10,0)$ and $h=h(128)$}
\label{fig:3d-par-nud}
\end{figure}

\begin{figure}[!htp]
         \centering
         \includegraphics[width=.91\textwidth]{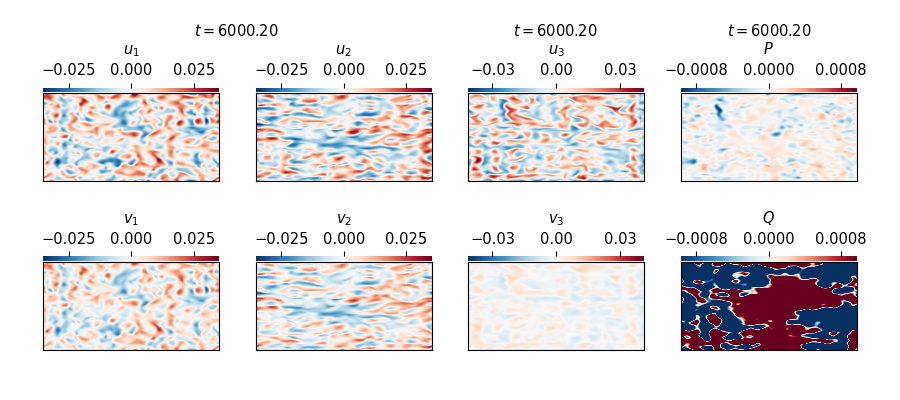}
         \includegraphics[width=.91\textwidth]{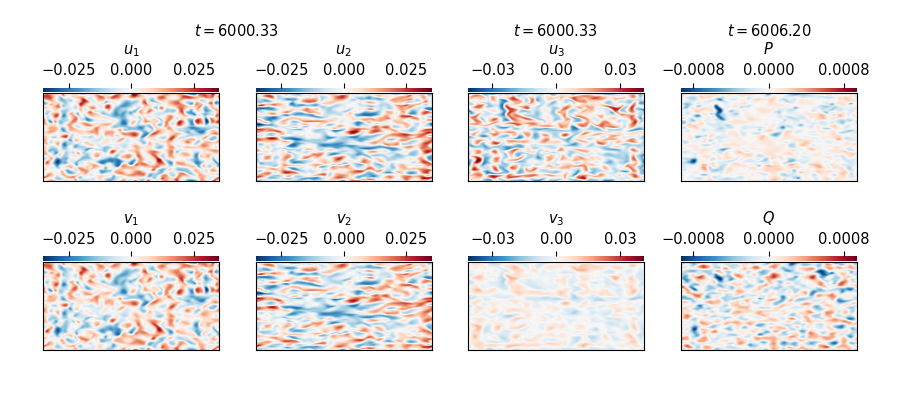}
         \includegraphics[width=.91\textwidth]{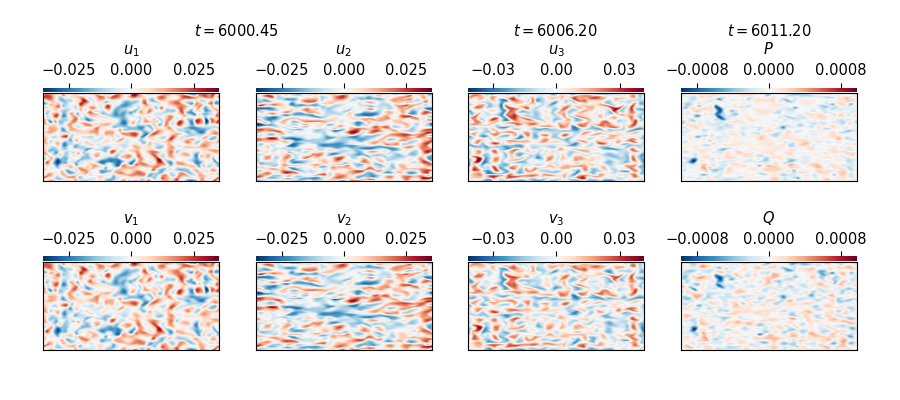}
\caption{Abridged data assimilation for the 3D Smagorinsky model with $(\mu_{1},\mu_{2},\mu_3)=(10,10,0)$ and $h=h(128)$. These are the slices on the mid-plane $(0,2\pi)\times(0,2\pi)\times\{z=\pi\}$.  Note the time progression is different for different components of velocity and pressure.}
         \label{fig:3d-da-abr}
\end{figure}

\clearpage

\section{Acknowledgments}   
The research of M.~Jolly was supported in part by NSF grant DMS-1818754. The authors acknowledge the Indiana University Pervasive Technology Institute (see \cite{stewart2017indiana}) for providing HPC (Big Red 3, Carbonate) and storage resources that have contributed to the research results reported within this paper.

\section{Appendix A}
\label{AppendixA}
In this appendix, we prove the monotonicity  property \eqref{T-mon}.   With $ A = (a_{ij})_{i,j=1}^3, B = (b_{ij})_{i,j=1}^3 \in \mathbb{R}_{sym}^{3 \times 3}$, consider 
\begin{equation}
\mathbf{T}(A) =  2\, \left(\nu_0  + \nu_1 |A|_F^{p-2} \right)\, A,  \quad p\geq 2.
\end{equation}
Motivated by
\begin{equation}\label{Lem1Eq1}
\mathbf{T}_{ij}(A) - \mathbf{T}_{ij}(B) = \int_0^1\frac{\rm{d}}{\rm{d} \tau} \mathbf{T}_{ij} \big(\tau A + (1 - \tau) B\big)  \, d\tau,
\end{equation}
we compute $\frac{\rm{d}}{\rm{d} \tau} \mathbf{T}_{ij}$  as 
\begin{equation}\label{Lem1Eq2}
\begin{split}
\frac{\rm{d}}{\rm{d} \tau} \mathbf{T}_{ij} \left(\tau A + (1 - \tau) B\right) & =  2\,  \frac{\rm{d}}{\rm{d} \tau} \left[  \left(\nu_0 + \nu_1 |\tau A+ (1 - \tau) B |^{p-2}_F\right)  \, \left( \tau A_{ij} + (1 - \tau) B_{ij}\right) \right]\\
& =2\,  \left[ \frac{\rm{d}}{\rm{d} \tau}\left(\nu_0 + \nu_1 |\tau A + (1 - \tau) B |^{p-2}_F \right)\right]  \,  \left( \tau a_{ij} + (1 - \tau) b_{ij}\right)\\
& +2\,   \left(\nu_0 + \nu_1 |\tau A + (1 - \tau) B|^{p-2}_F \right)  \, \left( a_{ij} - b_{ij}\right).
\end{split}
\end{equation}
 Working on the first term on the right-hand side above in more details, we obtain
\begin{equation}\label{Lem1Eq3}
\begin{split}
 \frac{\rm{d}}{\rm{d} \tau}\left(|\tau A+ (1 - \tau) B|^{p-2}_F \right)  & =   \frac{\rm{d}}{\rm{d} \tau} \left[ \sum_{i,j=1}^3 \left( \tau a_{ij} + (1-\tau ) b_{ij} \right)^2 \right]^{\frac{p-2}{2}}\\
 & = \frac{p-2}{2} \, |\tau A + (1 - \tau) B|^{p-4}_F\, \sum_{i,j=1}^3\, 2 \big(  a_{ij} -  b_{ij}\big) \, \big(  \tau a_{ij} + (1-\tau )b_{ij}\big) \\
 & = (p-2) \, |\tau A + (1 - \tau) B|^{p-4}_F\,  \, \sum_{k,l=1}^3\,  \big(  a_{kl} -  b_{kl}\big) \, \big(  \tau a_{kl} + (1-\tau )b_{kl}\big).\\
\end{split} 
\end{equation}
Now with the help of (\ref{Lem1Eq1}) we can write
\begin{equation}\label{Lem1Eq4}
\begin{split}
\mathbf{T}(A) - \mathbf{T}(B) : (A- B)  & = \sum_{i,j =1}^3 \big(\mathbf{T}_{ij}(A) - \mathbf{T}_{ij}(B) \big)  \, \big( a_{ij} -  b_{ij}\big)\\
& =  \sum_{i,j =1}^3 \left(\int_0^1\frac{\rm{d}}{\rm{d} \tau} \mathbf{T}_{ij} \big(\tau A+ (1 - \tau) B \big)  \, d\tau \right)  \, \big( a_{ij} -  b_{ij}\big).
\end{split}
\end{equation}
Plugging \eqref{Lem1Eq2} and \eqref{Lem1Eq3} in \eqref{Lem1Eq4}, we have
\begin{equation*}\label{Lem1Eq5}
\begin{split}
& \mathbf{T}(A) - \mathbf{T}(B) : (A - B)   =  \int_0^1  \sum_{i,j =1}^3 \frac{\rm{d}}{\rm{d} \tau} \mathbf{T}_{ij} \big(\tau A+ (1 - \tau) B \big) \left( a_{ij} -  b_{ij} \right)  \, d\tau \\
&  = 2\,  \int_0^1  \sum_{i,j=1}^3 \big(\nu_0 + \nu_1 |\tau A+ (1 - \tau) B|^{p-2}_F \big)  \, \big( a_{ij} - b_{ij}\big)^2 \, d\tau + 2\,  \int_0^1 \big[ \nu_1\, (p-2) \, |\tau A+ (1 - \tau)  B |^{p-4}_F \\
& \quad \times \sum_{i,j,k,l =1}^3  \big(  \tau a_{ij}+ (1-\tau ) b_{ij}\big)\, \big(  \tau a_{kl}+ (1-\tau ) b_{kl}\big)\, \big(a_{ij} - b_{ij} \big)\,   \big(a_{kl} - b_{kl} \big) \big]\, d\tau. 
\end{split}
\end{equation*}
Since  $\nu_1 |\tau A + (1 - \tau) B|^{p-2}_F \geq 0$, the first integral above can be bounded from below as follows
\begin{equation}
\begin{split}
 \int_0^1  \big[\sum_{i,j=1}^3 \big(\nu_0 + \nu_1 |\tau A+ (1 - \tau) B|^{p-2}_F \big)  \, \big( a_{ij} - b_{ij}\big)^2 \big]\, d\tau & \geq  \nu_0 \sum_{i,j=1}^3  \, \big( a_{ij} - b_{ij}\big)^2 \\
&=  \nu_0 \, |A - B|_F^2.
\end{split}
 \end{equation}
As to the second term, we notice that it is non-negative since
 \begin{equation}
 \begin{split}
&\sum_{i,j,k,l =1}^3  \big(  \tau a_{ij}+ (1-\tau ) b_{ij}\big)\, \big(  \tau a_{kl}+ (1-\tau ) b_{kl}\big)\, \big(a_{ij} - b_{ij} \big)\,   \big(a_{kl} - b_{kl} \big)\\
 & = \left[ \sum_{i,j=1}^3 \big(  \tau a_{ij} + (1-\tau ) b_{ij}\big)\,   \big(a_{ij} - b_{ij}\big) \right]^2 \geq 0. 
 \end{split}
 \end{equation}
 Therefore, we conclude that
 $$ \mathbf{T}(A) - \mathbf{T} (B) : ( A - B) \,  \geq 2\,  \nu_0\,   |A - B|_F^2,$$
 which proves \eqref{T-mon}.

\bibliographystyle{plain}
\bibliography{lady.bib}

\end{document}